\documentclass[12pt,a4paper]{article}

\usepackage[margin=1in]{geometry}
\usepackage{amsmath,amssymb,amsthm,enumerate}
\usepackage{tikz-cd}
\usepackage{xcolor}
\usepackage{hyperref}

\newtheorem{thm}{Theorem}[section]

\numberwithin{equation}{section}
\numberwithin{figure}{section}

\theoremstyle{plain}
\newtheorem{cor}[thm]{Corollary}
\newtheorem{prop}[thm]{Proposition}
\newtheorem{lem}[thm]{Lemma}

\theoremstyle{definition}
\newtheorem{defi}[thm]{Definition}
\newtheorem{ex}[thm]{Example}
\newtheorem{question}[thm]{Question}

\theoremstyle{remark}
\newtheorem{rem}[thm]{Remark}

\title{Transcendental Numerical Dimension and Rational Quotients in Kähler Geometry}
\author{Songchen Liu}
\date{August 2026}

\begin{document}

\maketitle

\begin{abstract}
We study properties of the transcendental numerical dimension on compact Kähler manifolds. In particular, we establish a connection between this invariant and degenerate divisors in fibrations, extending few  results from the projective setting.

We also study fibrations between compact Kähler manifolds with rationally connected general fibre. We prove that every pseudo-effective line bundle contained in a tensor power of the cotangent bundle comes from a pseudo-effective line bundle contained in the corresponding tensor power on the base, up to an explicit relation involving degenerate divisors. 

Finally, combining these results with a theorem of Cao--Păun, we answer a question they posed on rational quotients. More precisely, let $q : X\dashrightarrow Q$ be the rational quotient of a compact Kähler manifold $X$, and let $L$ be a pseudo-effective line bundle on $X$ admitting an injection
$L\to(\Omega_X^1)^{\otimes m},~m \geq 1$.  We prove that
\[
\nu(L,X)\leq\nu(K_Q,Q).
\]
To the best of our knowledge, both the descent theorem and this inequality are new even in the projective setting.

\end{abstract}

\section{Introduction}

A compact Kähler manifold \(X\) is called uniruled if it is covered by rational curves. Uniruled manifolds play a fundamental role in birational classification; see, for instance, \cite{Mor82,BCHM10}. A celebrated result of Miyaoka--Mori and Boucksom--Demailly--Păun--Peternell \cite{MM86,BDPP13} states that a projective manifold \(X\) is uniruled if and only if its canonical bundle \(K_X\) is not pseudo-effective. Very recently, Ou \cite{Ou25} extended this characterization to arbitrary compact Kähler manifolds.

A key ingredient in Ou's approach is an algebraicity criterion for foliations. Motivated by Ou's work, Cao--Păun \cite{CP26} obtained a generalization of this criterion. As an application, they proved that if $X$ is a compact Kähler manifold with \(K_X\) pseudo-effective, then every torsion-free quotient
\[
    (\Omega_X^1)^{\otimes m}\twoheadrightarrow\mathcal Q,
    ~~m\geq1,
\]
has pseudo-effective determinant $\det \mathcal Q$; see \cite[Theorem 1.4]{CP26}.

The numerical dimension is a fundamental invariant in algebraic geometry. It admits a transcendental definition for pseudo-effective classes on compact Kähler manifolds, which agrees with the classical algebraic definition in the projective case; see, for instance, Section \ref{sec 212}. 

The Cao--Păun's result has a consequence for numerical dimension.
Let \(X\) be a compact Kähler manifold with \(K_X\) pseudo-effective, and
suppose that a pseudo-effective line bundle \(L\) admits an injection
$L\to(\Omega_X^1)^{\otimes m}$.
After saturating the image of \(L\), the resulting quotient is torsion-free and hence has pseudo-effective determinant by the result above. Since
the saturation differs from \(L\) by an effective divisor, there exists a pseudo-effective line bundle \(P\) such that
\[
    c_mK_X = L + P,~c_m = m(\dim X)^{m-1}.
\]
The monotonicity of numerical dimension therefore yields the elegant
inequality
\begin{equation}\label{ineq 1.1}
    \nu(L,X)\leq \nu(K_X,X).
\end{equation}

Seeking an analogue without assuming that $K_X$ is pseudo-effective, Cao--Păun posed the following question \cite[Question~5.6]{CP26}.

\begin{question}\label{question 1.1}
Let \(q :  X\dashrightarrow Q\) be the rational quotient of a compact Kähler manifold \(X\). Suppose that a pseudo-effective line bundle \(L\) on \(X\) admits an injection $L\to(\Omega_X^1)^{\otimes m}$ for some \(m\geq1\). Does one have
\[
    \nu(L,X)\leq\nu(K_Q,Q)?
\]
\end{question}
\noindent We refer to Section~\ref{sec 2.2.2} for the definition and relevant properties of rational quotients. In particular, by choosing a suitable base, we may assume that \(Q\) is also a compact Kähler manifold.

This question is a natural extension of the above  inequality. Indeed, if \(K_X\) is pseudo-effective, then \(X\) is not uniruled, and the identity map
$\operatorname{id}_X :  X\to X$ is a rational quotient. Thus, (\ref{ineq 1.1}) is precisely the special case in which \(X\) itself is the base of its rational quotient.

To answer this question, we prove the following main Theorem.

\begin{thm}(=Theorem \ref{Main thm 1.2 in intro})
Let \(f :  X\to Q\) be a fibration between connected
compact Kähler manifolds whose general fibre is rationally connected,
and assume that \(K_Q\) is pseudo-effective. Let \(m\geq1\), and let \(L\) be a pseudo-effective
line bundle admitting an injective morphism
$ \varphi :  L\to(\Omega_X^1)^{\otimes m}$. 
Then
\[
    \nu(L,X)\leq\nu(K_Q,Q).
\]
\end{thm}

Let us explain why this theorem gives an affirmative answer to
Question~\ref{question 1.1}. Let \(q :  X\dashrightarrow Q\) be a
rational quotient. Its base \(Q\) is non-uniruled, and hence \(K_Q\) is
pseudo-effective by Ou's theorem. Choose a resolution of the indeterminacies of \(q\),
$ \mu :  \widetilde X\to X$, such that
\[
    f:=q\circ\mu : \widetilde X\to Q
\]
is a fibration. Then \(\widetilde X\) is compact Kähler and the general
fibre of \(f\) is rationally connected. Moreover, \(\mu^*L\) is pseudo-effective,
and \(\varphi\) induces an injection
$ \mu^*L\to(\Omega_{\widetilde X}^1)^{\otimes m}$.
Applying the theorem above and using the invariance of
numerical dimension under proper bimeromorphic morphism, we obtain
\[
    \nu(L,X)
    =\nu(\mu^*L,\widetilde X)
    \leq\nu(K_Q,Q),
\]
which proves the inequality in Question~\ref{question 1.1}.

The guiding idea behind the proof is simple: descend
$  L\to(\Omega_X^1)^{\otimes m}$ to a pseudo-effective line bundle
$  M\to(\Omega_Q^1)^{\otimes m}$,
establish an explicit relation between \(L\) and \(f^*M\), and then apply
\cite[Theorem~1.4]{CP26} to \(M\) on \(Q\) to derive the desired
inequality.

Carrying out this strategy requires two main ingredients. The first controls the effect
of certain divisors on numerical dimension, while the second establishes the
required descent.

Let $f : X\to Y$ be a fibration between compact Kähler manifolds. Various notions of degeneracy for vertical divisors recur throughout relative birational geometry; see, for instance, \cite{Bir12,GL13,Na04}. The notion relevant to our setting is that of an $f$-degenerate divisor.

An effective vertical $\mathbb R$-divisor $D$ on $X$ is called
$f$-degenerate if, for every prime divisor $P\subset Y$ contained in
$f(\operatorname{Supp}D)$, there exists a prime divisor
$\Gamma\subset X$ such that
\[
f(\Gamma)=P
~~\text{and}~~
\Gamma\not\subset\operatorname{Supp}D.
\]
In other words, $D$ never contains all the divisorial components lying
over a prime divisor contained in $f(\operatorname{Supp}D)$. In
particular, $f$-degenerate divisors may be viewed as a natural
generalization of divisors whose images under $f$ have codimension at
least 2 in $Y$.

The results of Section~\ref{Sec 3.2} show that \(f\)-degenerate divisors
carry only very limited positivity. In particular, adding such a divisor
to the pullback of a pseudo-effective class on \(Y\) does not increase its numerical dimension.

\begin{thm}(=Theorem \ref{thm:f-degenerate-main})
Let \(f :  X\to Y\) be a fibration between compact Kähler manifolds,
let \(\alpha\in H^{1,1}(Y,\mathbb R)\) be a pseudo-effective class, and let
\(D\geq0\) be an \(f\)-degenerate \(\mathbb R\)-divisor. Then
\[
    \nu(f^*\alpha+\{D\},X)
    =\nu(\alpha,Y).
\]
\end{thm}

The main tool in the proof is Boucksom's divisorial Zariski decomposition. By extending \cite[Corollary~III.5.3]{Na04} to the Kähler setting, we show that adding an effective
$f$-degenerate divisor does not change the positive part, i.e., $Z(f^*\alpha+\{D\})=Z(f^*\alpha).$ 
Since removing the negative part does not affect the numerical dimension, we obtain the desired formula.

We now turn to the second ingredient, where the relevance of \(f\)-degenerate divisors to our problem becomes apparent. We further assume that the general fiber of $f: X \to Y$ is rationally connected. Rational connectedness of the general fibre imposes a strong rigidity on tensorial differentials: a pseudo-effective line contained in \((\Omega_X^1)^{\otimes m}\) must come generically from the base. The following theorem promotes this generic descent to a global relation and gives precise control of the resulting divisorial discrepancy.

\begin{thm}(=Theorem \ref{thm:divisorial_descent_cotangent_line})
Let $f: X\to Y$ be a fibration between compact
Kähler manifolds whose general fibre is rationally connected. Let
$L$ be a pseudo-effective line bundle admitting an injection
$\varphi: L\to(\Omega_X^1)^{\otimes m},~m\geq1$.

Then there exist a pseudo-effective line bundle $M$ on $Y$, embedded as a saturated subsheaf of $(\Omega_Y^1)^{\otimes m}$, and effective divisors $D,E$ on $X$, where $D$ is
$f$-degenerate, such that
\begin{equation}\label{eq 1.2}
L\otimes\mathcal O_X(E)
\simeq
f^*M\otimes\mathcal O_X(D).
\end{equation}
\end{thm}

The construction of $M$ relies on covering families of very free
rational curves on the general fibres of $f$. Setting
$V:=(\Omega_Y^1)^{\otimes m}$, we first restrict $\varphi$ to these
curves to show that its image is generically contained in $f^*V$.
The resulting meromorphic map $\rho\colon X\dashrightarrow\mathbb P_Y(V)$
is then shown, using again that these curves cover a general fibre, to
be constant along the general fibres of $f$. It consequently factors
meromorphically through $f$ as
\[
    \rho=\bar\rho\circ f,~~
    \bar\rho\colon Y\dashrightarrow\mathbb P_Y(V),
\]
where $\bar\rho$ is a meromorphic section of $\mathbb P_Y(V) \to Y$. This allows us to construct a saturated rank-one subsheaf $M\subset V$ and then establish \eqref{eq 1.2} by comparing its
pullback with $L$.

The relation (\eqref{eq 1.2}) is precisely where the two ingredients meet. For any pseudo-effective line bundle $L \to (\Omega_X^1)^{\otimes m}$, we can find an unique $M_L \to (\Omega_Y^1)^{\otimes m}$. As in (\ref{eq 1.2}), the effectivity of \(E\), together with the
numerical invariance under adding the \(f\)-degenerate divisor \(D\), gives $\nu(L,X) \leq \nu(M_L,Y)$. Note that,  this Theorem does not assume that $K_Y$ is pseudo-effective.

\subsubsection*{Acknowledgements}
The author used the GPT-5.6 Plus version Sol model, for proof exploration and consistency checks during the development of this work. In particular, the author thanks GPT for helping to flesh out many of the ideas and technical details in Section \ref{sec 4}. AI generated suggestions were treated as provisional; the author assumes full responsibility for every argument and for the final text.

\subsubsection*{Conventions and notation}

In this paper, we use the following conventions and notation.

\begin{itemize}
    \item We use the convention
    $ dd^c=\frac{\sqrt{-1}}{\pi}\partial\bar\partial$.

    \item For a closed real $(p,q)$-form $\theta$, we denote its cohomology class by
    $ \{\theta\}\in H^{p,q}(X,\mathbb R)$.
    
    For a divisor $D$, we denote by $[D]$ the associated current of integration
    and by
    $\{D\} := c_1(D)\in H^{1,1}(X,\mathbb R)$ its cohomology class.

    \item A morphism $f :  X\to Y$ is called a \emph{fibration} if it is
    proper and surjective with connected fibres, we denote the fiber at the point $y$ by $X_y = f^{-1}(y)$. A \emph{section} of $f$ is a
    morphism $s :  Y\to X$ such that
    $  f\circ s=\operatorname{id}_Y$.

    \item For a $\mathbb R$-divisor $D$ and a prime divisor $F$, we denote by
${\rm ord}_F D$ the coefficient of $F$ in $D$. 
Inequalities between
divisors are understood coefficientwise: we write $D_1\geq D_2$ if
$
{\rm ord}_F D_1\geq {\rm ord}_F D_2$
for every prime divisor $F$. In particular, $D\geq0$ means that $D$ is effective.

    \item For a morphism $f :  X\to Y$ between complex manifolds,
we denote by $\operatorname{Crit}(f)$ its critical locus, so that $\operatorname{rank}(df_x)<\dim Y$  for  $x\in\operatorname{Crit}(f)$.
We define the \emph{submersion locus} of $f$ by $X^\circ :=f^{-1}\bigl(Y\backslash f(\operatorname{Crit}(f))\bigr)$.

    \item We use \emph{usc} as an abbreviation for \emph{upper semicontinuous}.

    \item We use the convention that $\mathbb P(E)$ parametrizes lines in $E$, rather than in its dual.

    \item We say that a property holds at a general point if it holds outside a proper analytic subset.
\end{itemize}

\section{Preliminaries}

Unless otherwise specified, in this section we generally assume that $X$ is a compact Kähler manifold of complex dimension $n$.

\subsection{Pluripotential Theory}

We first recall some basic notions from pluripotential theory. Let $M$ be a complex manifold of dimension $n$. A function $u:M\to [-\infty,+\infty)$
is called plurisubharmonic (psh) if it is usc and its restriction to every complex line in a local holomorphic coordinate chart is subharmonic or identically equal to $-\infty$. A function $u$ is called quasi-plurisubharmonic (qpsh) if locally it can be written as the sum of a psh function and a smooth function.

A $(1,1)$-current $T$ on $M$ is called positive if
$\langle T,\Phi\rangle\geq 0$ for every smooth compactly supported positive $(n-1,n-1)$-form $\Phi$. It is called a closed positive $(1,1)$-current if, in addition, $dT=0$.

Let $\theta$ be a smooth closed real $(1,1)$-form on $X$. By the $dd^c$-Lemma for currents, every closed positive $(1,1)$-current $T \in \{\theta\}$ can be written as
$T=\theta+dd^c u$
for some quasi-plurisubharmonic function $u$, unique up to an additive constant. We denote by
$$
\operatorname{PSH}(X,\theta) = \left\{ u ~\text{is qpsh function on}~X:
\theta+dd^c u\geq 0
\right\}.
$$
Thus the closed positive $(1,1)$-currents in the cohomology class $\{\theta\}$ are precisely the currents of the form
$\theta+dd^c u$, with $u\in\operatorname{PSH}(X,\theta)$.

We will frequently compare the singularities of currents. Given two qpsh functions $u$ and $v$, we say that $u$ is less singular than $v$, and write
$u\succeq v$, if $v\leq u+O(1)$ on $X$. We say that $u$ and $v$ have the same singularity type, and write
$u\simeq v$, if $u\succeq v$ and $v\succeq u$. Similarly, for closed positive $(1,1)$-currents
$T_i=\theta_i+dd^c\varphi_i$, we write $T_1 \succeq T_2$ (resp. $T_1 \simeq T_2$) if $\varphi_1 \succeq \varphi_2$ (resp. $\varphi_1 \simeq \varphi_2$).

\subsubsection{Positivity and Intersection Products of Transcendental Classes}


A class $\alpha=\{\theta\}\in H^{1,1}(X,\mathbb R)$ is called
\emph{pseudoeffective} (psef for short) if it contains a closed positive $(1,1)$-current, or
equivalently if $\operatorname{PSH}(X,\theta)\neq\varnothing$. It is called \emph{big} if it
contains a Kähler current, i.e., $\alpha$ is big if there exist
$\varphi\in\operatorname{PSH}(X,\theta)$ and $\varepsilon>0$ such that
\[
    \theta+dd^c\varphi\geq\varepsilon\omega.
\]

Suppose now that $\alpha=\{\theta\}$ is psef. The envelope
\[
    V_\theta
    :=\left(\sup\bigl\{u\in\operatorname{PSH}(X,\theta):u\leq 0\text{ on }X\bigr\}\right)^*
\]
has minimal singularities in $\operatorname{PSH}(X,\theta)$, where ${}^*$ denotes the upper
semicontinuous regularization. 
More generally, $u\in\operatorname{PSH}(X,\theta)$ is said to have
minimal singularities if $u\succeq v$ for every
$v\in\operatorname{PSH}(X,\theta)$.

Let $1\leq p\leq n$. For $j=1,\ldots,p$, let $\theta_j$ be a smooth closed real
$(1,1)$-form whose class $\alpha_j:=\{\theta_j\}$ is big, and let
\[
    \varphi_j\in\operatorname{PSH}(X,\theta_j),~~
    T_j:=\theta_j+dd^c\varphi_j.
\]
Set $V_j:=V_{\theta_j}$ and, for $\ell\in\mathbb N$, define the canonical
truncations and the associated plurifine open sets by
$$ 
\varphi_j^{(\ell)}:=\max\{\varphi_j,V_j-\ell\}, 
~~~ U_\ell:=\bigcap_{j=1}^p\{\varphi_j>V_j-\ell\}.
$$
Each $\varphi_j^{(\ell)}$ has minimal singularities. The corresponding
Bedford--Taylor product is defined on the common ample locus
$\bigcap_j\operatorname{Amp}(\alpha_j)$ and extended trivially across its
complement. Here $\operatorname{Amp}(\alpha_j)$ is a Zariski open set constructed by Boucksom \cite{Bou04} such that $V_j \in L^\infty_{loc}({\rm Amp}(\alpha_j))$. The currents
\[
    R_\ell
    :=\mathbf 1_{U_\ell}
      \bigwedge_{j=1}^p
      \left(\theta_j+dd^c\varphi_j^{(\ell)}\right)
\]
form an increasing sequence and converge weakly to a closed positive
$(p,p)$-current. Its limit is called the \emph{non-pluripolar product} of
$T_1,\ldots,T_p$ and is denoted by
\[
    \left\langle T_1\wedge\cdots\wedge T_p\right\rangle
    :=\lim_{\ell\to\infty}R_\ell.
\]
This construction is independent of all the auxiliary choices and the resulting
current does not charge pluripolar sets; see \cite[Section~2]{DDL25}.

We shall repeatedly use the monotonicity of non-pluripolar masses.
\begin{thm}\label{DDL25 thm 33}(\cite[Theorem~3.3]{DDL25})
Let $\theta_1,\ldots,\theta_n$ be smooth closed real
$(1,1)$-forms whose cohomology classes are big, and let
$u_j,v_j\in\operatorname{PSH}(X,\theta_j)$. Set $T_j = \theta_j + dd^c u_j, S_j = \theta_j + dd^c v_j$, if $ T_j \succeq S_j$, i.e., 
$v_j\leq u_j+O(1)$ for every $j$, then
$$
    \int_X\left\langle
        T_1\wedge\cdots\wedge T_n
    \right\rangle
    \geq
    \int_X\left\langle
        S_1\wedge\cdots\wedge  S_n
    \right\rangle.
$$
\end{thm}

The non-pluripolar product is symmetric and multilinear on the cone of closed
positive $(1,1)$-currents \cite[Proposition~1.4]{BEGZ10}. In particular, if
$T,S,T_2,\ldots,T_p$ are closed positive $(1,1)$-currents whose cohomology classes are big and $a,b\geq0$, then
\[
\begin{split}
    \left\langle (aT+bS)\wedge T_2\wedge\cdots\wedge T_p\right\rangle={}
    a\left\langle T\wedge T_2\wedge\cdots\wedge T_p\right\rangle
    +b\left\langle S\wedge T_2\wedge\cdots\wedge T_p\right\rangle.
\end{split}
\]


Let $\alpha$ be big and let $0\le S\in\alpha$ be smooth outside a proper analytic subset
$A\subset X$. Then its non-pluripolar product can be viewed as the classical wedge product of smooth forms on $X\backslash A$, followed by trivial
extension across $A$, that is  $\langle S^k\rangle =\mathbf 1_{X\backslash A}S^k$.

Let $\alpha_1,\ldots,\alpha_p$ be big $(1,1)$-classes. For each $j$, choose a
closed positive current $T_{j,\min}\in\alpha_j$ with minimal singularities. The
class
\[
    \left\langle\alpha_1\cdots\alpha_p\right\rangle
    :=
    \left\{
        \left\langle
            T_{1,\min}\wedge\cdots\wedge T_{p,\min}
        \right\rangle
    \right\}
    \in H^{p,p}(X,\mathbb R)
\]
is independent of the chosen currents with minimal singularities. It is called
the \emph{positive intersection product} of $\alpha_1,\ldots,\alpha_p$; see
\cite[Definition~1.17]{BEGZ10}. This agrees with the movable intersection
product introduced in \cite[Section~3]{BDPP13}. If $\alpha_1,\ldots,\alpha_p$ are merely psef, fix an arbitrary
Kähler class $\beta$. Following the big-class approximation of
\cite{BDPP13}, one defines
\[
    \left\langle\alpha_1\cdots\alpha_p\right\rangle
    :=\lim_{\varepsilon\downarrow0}
      \left\langle
        (\alpha_1+\varepsilon\beta)\cdots
        (\alpha_p+\varepsilon\beta)
      \right\rangle.
\]
The limit exists in $H^{p,p}(X,\mathbb R)$ and is independent of the Kähler
class $\beta$. In particular, for a psef class
$\alpha$, we have $ \left\langle\alpha^p\right\rangle
    =\lim_{\varepsilon}
      \left\langle(\alpha+\varepsilon\beta)^p
      \right\rangle$.

\subsubsection{Definition of the Numerical Dimension of Transcendental Classes}\label{sec 212}

Let $\alpha\in H^{1,1}(X,\mathbb R)$ be a psef class. Using the positive
intersection products introduced above, we define the \emph{numerical
dimension} of $\alpha$ by
\[
    \nu_{\mathrm{np}}(\alpha,X)
    :=\max\bigl\{
        k\in\{0,\ldots,n\}:\langle\alpha^k\rangle\neq0
      \bigr\}.
\]
Here, $\langle\alpha^0\rangle:=1$. Since $\langle\alpha^k\rangle$ is a closed
positive $(k,k)$-current, its cohomology class is nonzero if it has
positive mass. Therefore,
\[
    \nu_{\mathrm{np}}(\alpha,X)
    =\max\left\{
        k\in\{0,\ldots,n\}:
        \int_X\langle\alpha^k\rangle\wedge\omega^{n-k}>0
      \right\}.
\]
If $\alpha$ is not psef, we set $\nu_{\mathrm{np}}(\alpha,X):=-\infty$. Note that $\alpha$ is big if and only if $\nu_{\rm np}(\alpha,X) = n$, see \cite{Bou02}.

For each $\varepsilon>0$, we define $\alpha[-\varepsilon\omega]
    :=\bigl\{T\in\alpha:T\geq-\varepsilon\omega\bigr\}$,
and denote by $\alpha[-\varepsilon\omega]_{\mathrm{an}}$ the subset consisting
of currents with analytic singularities. Recall that a current
$T=\theta+dd^c\varphi$ has \emph{analytic singularities} if, locally,
\[
    \varphi
    =c\log\!\left(\sum_{\ell=1}^N|f_\ell|^2\right)+v,
    ~~ c>0,
\]
where the $f_\ell$ are holomorphic functions and $v$ is smooth. In particular,
$T$ is smooth on $X\backslash \operatorname{Sing}(T)$, so $T$ is smooth outside a proper analytic subset.

We recall the following Demailly's regularization Theorem.

\begin{thm}\label{thm:DP04-regularization}(\cite[Theorem~3.2]{DP04})
Let $T=\theta+dd^c\varphi$ be a closed real $(1,1)$-current satisfying
$T\geq\gamma$, where $\theta$ is smooth and $\gamma$ is a continuous real
$(1,1)$-form. Then there exist quasi-psh functions $\varphi_j$ with analytic
singularities and numbers $\delta_j\downarrow0$ such that
\[
    T_j:=\theta+dd^c\varphi_j\geq\gamma-\delta_j\omega,
    ~~
    \varphi_j\searrow\varphi.
\]
Moreover, $T_j$ is smooth on $X\backslash  Z_j$, where $(Z_j)_j$ is an
increasing sequence of analytic subsets of $X$. In particular, if
$T\in\alpha[-\varepsilon\omega]$, then
$T_j\in\alpha[-(\varepsilon+\delta_j)\omega]_{\rm an}$.
\end{thm}

  By this Theorem, for $\alpha \in H^{1,1}(X,\mathbb R)$ psef, the class
$\alpha[-\varepsilon\omega]_{\mathrm{an}}$ is nonempty for every
$\varepsilon>0$. So we can now give an equivalent definition of the numerical dimension:
\[
    \nu_{\mathrm{an}}(\alpha,X)
    :=
    \max\left\{
        k\in\{0,\ldots,n\}:
        \limsup_{\varepsilon\downarrow0}\,
        \sup_{T\in\alpha[-\varepsilon\omega]_{\mathrm{an}}}
        \int_{X\backslash \operatorname{Sing}(T)}
            (T+\varepsilon\omega)^k\wedge\omega^{n-k}>0
      \right\}.
\]

Let us explain the equivalence of these two definitions. Write
$\alpha=\{\theta\}$ and set
\[
    \theta_\varepsilon:=\theta+\varepsilon\omega,
    ~~
    T_\varepsilon
    :=\theta_\varepsilon+dd^cV_{\theta_\varepsilon}.
\]
Thus $T_\varepsilon$ is a positive current with minimal singularities in the
big class $\alpha+\varepsilon\{\omega\}$. By the definition of the positive
product $\langle\alpha^k\rangle$, we have 
$$
\lim_{\varepsilon\downarrow0}
        \int_X\langle T_\varepsilon^k\rangle
        \wedge\omega^{n-k}>0
$$
for all $k \leq \nu_{\rm np}(\alpha,X)$. Fix $\varepsilon>0$ and let
$S'_\varepsilon\in\alpha[-\varepsilon\omega]_{\mathrm{an}}$. Set
$ S_\varepsilon
    :=S'_\varepsilon+\varepsilon\omega
    =\theta_\varepsilon+dd^c\varphi_\varepsilon\geq0$. 
Since $V_{\theta_\varepsilon}$ has minimal singularities in $\{\theta_\varepsilon\}$, we have $ \varphi_\varepsilon
   \preceq V_{\theta_\varepsilon}$.
The monotonicity of total non-pluripolar masses given by
Theorem~\ref{DDL25 thm 33} therefore yields
\[
\begin{split}
    \int_X\langle T_\varepsilon^k\rangle\wedge\omega^{n-k}
    \geq
    \int_X\langle S_\varepsilon^k\rangle\wedge\omega^{n-k} =
    \int_{X\backslash \operatorname{Sing}(S_\varepsilon)}
        S_\varepsilon^k\wedge\omega^{n-k}.
\end{split}
\]
Taking the supremum over
$S'_\varepsilon\in\alpha[-\varepsilon\omega]_{\mathrm{an}}$ and then the
limit superior as $\varepsilon\downarrow0$, we obtain
\[
    \nu_{\mathrm{np}}(\alpha,X)
    \geq\nu_{\mathrm{an}}(\alpha,X).
\]

For the converse inequality, apply
Theorem~\ref{thm:DP04-regularization} to $T_\varepsilon$. We may choose an
index $j_\varepsilon$ such that $T_{\varepsilon,j_\varepsilon}
    :=\theta_\varepsilon+dd^c\psi_\varepsilon
    \geq-\varepsilon\omega$,
where $\psi_\varepsilon$ has analytic singularities and
$\psi_\varepsilon\geq V_{\theta_\varepsilon}$. Set
\[
    G'_\varepsilon
    :=T_{\varepsilon,j_\varepsilon}-\varepsilon\omega
    \in\alpha[-2\varepsilon\omega]_{\mathrm{an}},~~
    G_\varepsilon
    :=T_{\varepsilon,j_\varepsilon}+\varepsilon\omega
    =G'_\varepsilon+2\varepsilon\omega\geq0.
\]
Since $\psi_\varepsilon \geq V_{\theta_\varepsilon}$, Theorem~\ref{DDL25 thm 33} gives
\[
\begin{split}
    \int_{X\backslash \operatorname{Sing}(G_\varepsilon)}
        G_\varepsilon^k\wedge\omega^{n-k}
    &=
    \int_X\langle G_\varepsilon^k\rangle\wedge\omega^{n-k} \\
    &\geq
    \int_X
        \left\langle(T_\varepsilon+\varepsilon\omega)^k\right\rangle
        \wedge\omega^{n-k} \geq
    \int_X\langle T_\varepsilon^k\rangle\wedge\omega^{n-k}.
\end{split}
\]
The last inequality follows from the expansion of the non-pluripolar product
after adding the smooth Kähler form $\varepsilon\omega$. Since
$G'_\varepsilon$ is admissible in the definition of
$\nu_{\mathrm{an}}(\alpha,X)$ at the parameter $2\varepsilon$, it follows that
\[
    \nu_{\mathrm{an}}(\alpha,X)
    \geq\nu_{\mathrm{np}}(\alpha,X).
\]
Consequently $    \nu_{\mathrm{an}}(\alpha,X)    =\nu_{\mathrm{np}}(\alpha,X)$.
We denote their common value by $\nu(\alpha,X)$.

\begin{rem}
When $X$ is projective and $\alpha=c_1(L)$, it follows from
\cite[Theorem~1.1]{Leh13} that $\nu(\alpha,X)$ agrees with the numerical
dimension $\nu(L,X)$ defined in algebraic geometry.
\end{rem}

Now let us give two fundamental properties of numerical dimension.
\begin{prop}\label{prop:basic-properties-numerical-dimension}
Let
$\alpha,\beta\in H^{1,1}(X,\mathbb R)$ be psef classes. Then we have:
\begin{enumerate}
    \item $ \nu(c\alpha,X)=\nu(\alpha,X)$, $\forall c >0$;

    \item $   \nu(\alpha,X)\leq\nu(\beta+\alpha,X)$.
\end{enumerate}
\end{prop}

\begin{proof}
By the homogeneity of the non-pluripolar product, we can directly obtain the first assertion. The second assertion follows from
\cite[Lemma~6.10]{CP26}.
\end{proof}

\subsubsection{Pullbacks of Currents}

Let $g :  X^n\to Y^m$ be a fibration between compact Kähler manifolds, and set $r:=n-m$.
Let $\Gamma_g\subset X\times Y$ be the graph of $g$, let $\pi_X$ and
$\pi_Y$ be the two projections, and set
\[
    \rho:=\pi_Y|_{\Gamma_g} : \Gamma_g\to Y.
\]

Let $R$ be a closed positive $(q,q)$-current on $Y$, and let $U\subset Y$
be the flat locus of $g$, i.e., $g|_{g^{-1}(U)} : g^{-1}(U) \to U$ is flat. Choose a sufficiently fine open covering
$(U_i)_i$ of $U$. After shrinking each $U_i$ if necessary, choose smooth
closed positive forms $R_{i,j}$ on $U_i$ converging weakly to $R|_{U_i}$,
and set
\[
    Q_i(R)
    :=
    \lim_{j\to\infty}
    \pi_Y^*R_{i,j}\wedge[\Gamma_i],
    ~~
    \Gamma_i:=\Gamma_g\cap(X\times U_i).
\]
By the local construction of Dinh--Sibony, these limits exist, are
independent of the regularizations, and $Q_i(R)$, $Q_j(R)$ agree on overlaps $\Gamma_i \cap \Gamma_j$; see
\cite[Section~3]{DS07}. They therefore glue to a closed positive current
$Q_U(R)$ on $\Gamma_U:=\rho^{-1}(U)$, and one sets
\[
    g_U^*(R|_U):=(\pi_X)_*Q_U(R).
\]

Set
\[
    C:=
    \left\{
        z\in\Gamma_g:
        \dim_z\rho^{-1}\bigl(\rho(z)\bigr)>r
    \right\},
    ~~
    \Gamma_g^\circ:=\Gamma_g\backslash  C.
\]
The map $\rho^\circ:=\rho|_{\Gamma_g^\circ}$ has locally only empty or pure $r$-dimensional fibres. Hence the same
construction defines a closed positive current $(\rho^\circ)^*R$ on
$\Gamma_g^\circ$. By \cite[Proposition~5.1]{DS07}, this current has locally
finite mass near $C$, and its trivial extension to $\Gamma_g$ is closed.
The \emph{Dinh--Sibony strict transform} of $R$ is defined by
\[
    g^\star R
    :=
    (\pi_X)_*\widetilde{(\rho^\circ)^*R}.
\]
Its restriction over $U$ coincides with $g_U^*(R|_U)$.

We shall use the following consequences of
\cite[Proposition~5.1]{DS07}.

\begin{prop}\label{prop:DS-pullback}
Let $R$ be a closed positive $(q,q)$-current on $Y$, where
$0\leq q\leq m$. Then $g^\star R$ is a closed positive $(q,q)$-current
on $X$, and
\[
    \int_X g^\star R\wedge\Omega^{n-q}
    \leq
    C_q\int_Y R\wedge\omega^{m-q},
\]
where $\Omega$ and $\omega$ are fixed Kähler forms on $X$ and $Y$,
respectively. If $R$ charges no proper analytic subset of $Y$, then
$g^\star R$ charges no proper analytic subset of $X$.
\end{prop}

We now provide some descriptions of this pullback.

\begin{lem}\label{prop:fixed-current-pullback}
Let $S$ be a closed positive $(1,1)$-current on $Y$ whose cohomology class is big and which is smooth outside a proper analytic set $A\subset Y$. Then,
for every $0\leq q\leq m$, $\left\langle(g^*S)^q\right\rangle
    =
    g^\star\left\langle S^q\right\rangle$. In particular,   $\left\langle(g^*S)^q\right\rangle=0$, for $m<q\leq n$.
\end{lem}

\begin{proof}
Let $\Delta_g\subset Y$ be the set of critical values of $g$, and set
\[
    Y^\circ:=Y\backslash (A\cup\Delta_g),
    ~~
    X^\circ:=g^{-1}(Y^\circ).
\]
Write
\[
    g^\circ:=g|_{X^\circ} :  X^\circ\to Y^\circ.
\]
On $Y^\circ$, the current $S$ is smooth and $g^\circ$ is a submersion.
The locality of the Dinh--Sibony construction therefore gives
\[
    \left.
    g^\star\left\langle S^q\right\rangle
    \right|_{X^\circ}
    =
    (g^\circ)^*\bigl((S|_{Y^\circ})^q\bigr)
    =
    \bigl((g^\circ)^*(S|_{Y^\circ})\bigr)^q.
\]
This is also the restriction of
$\left\langle(g^*S)^q\right\rangle$ to $X^\circ$ because $g^*S$ is smooth on $X^\circ$.

It remains to compare the extensions across $X\backslash  X^\circ$. The
current $\langle S^q\rangle$ charges neither $A$ nor
$\Delta_g$. By the no-mass property of the pure-fibre pullback
\cite[Theorem~1.1]{DS07}, its pullback on $\Gamma_g^\circ$ charges no mass
over $A\cup\Delta_g$. Since $g^\star$ is obtained by trivial extension
across the non-pure locus, it follows that
$ g^\star\left\langle S^q\right\rangle$
charges no mass on $X\backslash  X^\circ$. The same holds for
$\left\langle(g^*S)^q\right\rangle$. Hence both currents are
the trivial extension of the same current on $X^\circ$, proving the first
identity.

\end{proof}

\begin{lem}\label{lem:fibre-integration}
Let $R$ be a closed positive $(q,q)$-current on $Y$ which charges no
proper analytic subset. Then
\[
    \int_X
        g^\star R\wedge g^*\omega^{m-q}\wedge\Omega^r
    =
    c_g\int_Y R\wedge\omega^{m-q},
\]
where $c_g:=\int_{X_y}\Omega^r>0$ for every regular fibre $X_y$ of $g$.
\end{lem}

\begin{proof}
The closed positive $(0,0)$-current $g_*(\Omega^r)$ is constant, and $g_*(\Omega^r)  =  \int_{X_y}\Omega^r   = c_g $ on every regular fibre. The projection formula gives the identity over
the regular-value locus of $g$. Since neither $R$ nor $g^\star R$ charges
the analytic subsets, the identity holds globally.
\end{proof}

\subsubsection{Boucksom's Divisorial Zariski Decomposition}

Let $T$ be a closed real $(1,1)$-current on $X$ satisfying
   $ T\geq\gamma$
for some closed smooth real $(1,1)$-form $\gamma$. Locally, one can write 
$ T=dd^c\varphi$, 
where $\varphi + h$ is plurisubharmonic, $h$ is the local potential of $-\gamma$. The \emph{Lelong number} of
$T$ at a point $x\in X$ is defined by $\nu(T,x):=\nu(\varphi,x)$,
where
\[
    \nu(\varphi,x)
    :=\liminf_{z\to x}\frac{\varphi(z)}{\log|z-x|}
    =\sup\left\{
        c\geq0:\varphi(z)\leq c\log|z-x|+O(1)
      \right\}.
\]
This definition is independent of the choice of local potential and local
coordinates.

If $Y\subset X$ is an irreducible analytic subset, the \emph{generic Lelong
number} of $T$ along $Y$ is defined by
\[
    \nu(T,Y):=\inf_{x\in Y}\nu(T,x).
\]
By Siu's Theorem \cite{Siu74}, this number is equal to $\nu(T,x)$ for a general point
$x\in Y$. In particular, if $ E:=\sum_i a_iE_i$ 
is an effective $\mathbb R$-divisor and
$[E]=\sum_i a_i[E_i]$ is its current of integration, then, for every prime
divisor $D\subset X$,
\[
    \nu([E],D)=\operatorname{ord}_D(E).
\]

By Siu's decomposition \cite{Siu74}, every closed positive $(1,1)$-current $T$
admits a unique decomposition
\[
    T=R+\sum_D\nu(T,D)[D],
\]
where the sum extends over all prime divisors $D\subset X$ and converges
weakly, while $R$ is a closed positive $(1,1)$-current satisfying $\nu(R,D)=0$ for every prime divisor $D\subset X$. More generally, the same decomposition
holds when $T$ is a closed real $(1,1)$-current satisfying $T\geq\gamma$ for
some smooth real $(1,1)$-form $\gamma$; in this case, the residual current
satisfies $R\geq\gamma$. See \cite[Section~2.2.1]{Bou04}.

A psef class $\alpha$ is called \emph{modified nef} if, for every
$\varepsilon>0$, there exists a current
$T_\varepsilon\in\alpha[-\varepsilon\omega]$ such that
$\nu(T_\varepsilon,D)=0$ 
for every prime divisor $D\subset X$; see
\cite[Definition~2.2]{Bou04}.

Now let $\alpha\in H^{1,1}(X,\mathbb R)$ be a psef class. For each
$\varepsilon>0$, let $T_{\min,\varepsilon}$ be a current with minimal
singularities in $\alpha[-\varepsilon\omega]$. For every prime divisor
$D\subset X$, Boucksom defines the \emph{minimal multiplicity} of $\alpha$
along $D$ by
\[
    \nu(\alpha,D)
    :=\sup_{\varepsilon>0}\nu(T_{\min,\varepsilon},D)
    =\lim_{\varepsilon\downarrow0}
       \nu(T_{\min,\varepsilon},D).
\]
From the definitions of minimal multiplicity and Lelong number, for any positive closed current $T \in \alpha$ and prime divisor $D \subset X$, we have $T \preceq T_{{\rm min},\varepsilon}$, $\forall \varepsilon >0$, hence
\begin{equation}\label{lelong number ineq}
\nu(T,D) \geq \nu(\alpha,D),~~\forall~ {\rm prime~ divisor}~ D \subset X.
\end{equation}

The \emph{negative part} of $\alpha$ is defined by
$ N(\alpha):=\sum_D\nu(\alpha,D)D$.
Only finitely many coefficients in this sum are nonzero, and hence
$N(\alpha)$ is an effective $\mathbb R$-divisor; see
\cite[Theorem~3.12]{Bou04}. The \emph{Zariski projection}, or positive part,
of $\alpha$ is defined by
$ Z(\alpha):=\alpha-\{N(\alpha)\}$.

The decomposition
\[
    \alpha=Z(\alpha)+\{N(\alpha)\}
\]
is called the \emph{divisorial Zariski decomposition} of $\alpha$. By inequality (\ref{lelong number ineq}) and Siu's decomposition \cite{Siu74}, for $0 \leq T \in \alpha$, we have 
$0 \leq T - [N(\alpha)] \in Z(\alpha)$, hence $Z(\alpha)$ is psef. Moreover, the class
$Z(\alpha)$ is modified nef; see
\cite[Proposition~3.8]{Bou04}.

\begin{rem}
For a pseudo-effective $\mathbb R$-divisor $D$ on a projective manifold,
Nakayama defined the negative part $N_{\sigma}(D)$ in
\cite[Chapter~III, \S1]{Na04}. By
\cite[Theorem~5.4 and Appendix~A.1]{Bou04}, it is coincide with the definition of the negative part given by Boucksom, i.e., $N(\{D\})=N_{\sigma}(D)$.
\end{rem}

\subsubsection{Pluripotential Theory on Singular Analytic Spaces}
Throughout this subsection, $Z$ denotes a reduced irreducible complex analytic space. We recall the notions of smooth forms and plurisubharmonic functions on $Z$ that will be used below; see, for example, \cite{FN80,Dem85}.

A function $f:Z\to\mathbb R$ is called smooth if every point of $Z$ admits a neighbourhood $U$, a local embedding
\[
    \iota:U\hookrightarrow\Omega\subset\mathbb C^N,
\]
and a function $\widetilde f\in\mathcal C^\infty(\Omega)$ such that $f|_U=\widetilde f\circ\iota$.
We denote the resulting sheaf of smooth functions on $Z$ by $\mathcal C_Z^\infty$. Let $\mathcal{PH}_Z$ be the subsheaf of real-valued pluriharmonic functions, namely, smooth functions which are locally the real parts of holomorphic functions.

A smooth closed real $(1,1)$-form $\theta$ on $Z$ is, by definition, a global section of $  \mathcal C_Z^\infty/\mathcal{PH}_Z$.
More explicitly, $\theta$ is represented by an open covering $(U_i)_i$ of $Z$ and smooth functions
\[
    \rho_i\in\mathcal C_Z^\infty(U_i)
\]
such that
\[
    \rho_i-\rho_j\in\mathcal{PH}_Z(U_i\cap U_j).
\]
We then write
\[
    \theta|_{U_i}=dd^c\rho_i
\]
and call $\rho_i$ a local potential of $\theta$.

A smooth closed real $(1,1)$-form $\theta$ is called a Kähler form if every point of $Z$ admits a local embedding as above for which a local potential of $\theta$ is the restriction of a smooth strictly plurisubharmonic function on the ambient Euclidean space. If $p:Y\to Z$ is a holomorphic map, the pull-back $p^*\theta$ is the smooth closed real $(1,1)$-form represented on $p^{-1}(U_i)$ by the local potentials
\[
    \rho_i\circ p.
\]

We next recall the notion of plurisubharmonicity on a complex analytic space. Let $W$ be a reduced pure dimensional complex analytic space. An usc function  $u:W\to [-\infty,\infty)$
is called plurisubharmonic if $u\not\equiv-\infty$ and every point of $W$ admits a neighbourhood $U$, a local embedding $\iota:U\hookrightarrow\Omega\subset\mathbb C^N$,
and a function $\widetilde u\in\operatorname{PSH}(\Omega)$ such that
\[
    u|_U=\widetilde u\circ\iota.
\]
We denote the space of such functions by $\operatorname{PSH}(W)$.

The following Theorem of Fornæss and Narasimhan gives an intrinsic characterization of this definition.

\begin{thm}\label{FN80 thm 531}
    An usc function $u:Z\to [-\infty,\infty)$
    is plurisubharmonic if and only if, for every holomorphic map
    $f:\Delta\to Z$
    from the unit disc, the function $f^*u$ is subharmonic or identically equal to $-\infty$.
\end{thm}

\begin{proof}
    See \cite[Section 5.3]{FN80}.
\end{proof}

In particular, if $p:Y\to Z$ is holomorphic and $u\in\operatorname{PSH}(Z)$, then
$p^*u$ is plurisubharmonic on $Y$, unless it is identically equal to $-\infty$.

Now let $\theta$ be a smooth closed real $(1,1)$-form on $Z$, represented by local potentials $(U_i,\rho_i)$. We define
\[
    \operatorname{PSH}(Z,\theta)
    :=
    \left\{
        u:Z\to[-\infty,\infty):
        \begin{array}{l}
        u\not\equiv-\infty,\\
        u+\rho_i\in\operatorname{PSH}(U_i)
        \text{ for every }i
        \end{array}
    \right\}.
\]
Since two local potentials of $\theta$ differ by a pluriharmonic function, this definition is independent of the chosen local potentials. For
$u\in\operatorname{PSH}(Z,\theta)$, we use the notation
\[
    \theta+dd^cu\geq0.
\]

Finally, let $p:Y\to Z$ be a surjective morphism and let
$u\in\operatorname{PSH}(Z,\theta)$. Since
\[
    (u+\rho_i)\circ p
    =
    u\circ p+\rho_i\circ p \not\equiv -\infty,~~\forall i,
\]
 Theorem \ref{FN80 thm 531} shows that $p^*u\in\operatorname{PSH}(Y,p^*\theta)$.

\subsection{Algebraic Geometry}

\subsubsection{Saturated Subsheaves of Vector Bundles}

Let $X$ be a complex manifold, $E$ be a holomorphic vector bundle on $X$, and let
$\mathcal F\subset E$ be a coherent subsheaf. Denote by
$\operatorname{Tor}(E/\mathcal F)$ the torsion subsheaf of $E/\mathcal F$.
The \emph{saturation} of $\mathcal F$ in $E$ is
\[
    \mathcal F^{\mathrm{sat}}
    :=
    \ker\bigl(
        E\to
      (E/ \mathcal F)\big/ \operatorname{Tor}(E/\mathcal F)
    \bigr).
\]
Thus $\mathcal F$ is saturated if and only if $E/\mathcal F$ is
torsion-free. If $\mathcal F$ is saturated of rank one, then
$\mathcal F$ is reflexive, see \cite[Proposition 5.5.22]{Kob87}, hence it is locally free because $\mathcal{F}$ is rank one and $X$ smooth.

\begin{lem}\label{lem:sat_eff}
    Let $L$ be a line bundle on $X$, let $E$ be a holomorphic vector
    bundle, and let
    $\varphi :  L\to E$ be a nonzero morphism. Set $\mathcal L:=\varphi(L)\subset E$ and
    denote by $\mathcal L^{\mathrm{sat}}$ its saturation in $E$. Then $\mathcal L^{\mathrm{sat}}-L$ is effective.
\end{lem}

\begin{proof}
    Since $L$ is a line bundle and $\varphi$ is nonzero, $\varphi$ is
    injective, so $\mathcal L\simeq L$. Moreover,
    $\mathcal L^{\mathrm{sat}}$ is a saturated rank-one subsheaf of the
    vector bundle $E$. 

   Let
\[
    q :  E\to E/\mathcal L~~\text{and}~~
    \pi :  E/\mathcal L\to
    (E/\mathcal L)/\operatorname{Tor}(E/\mathcal L)
\]
be the natural quotient morphisms. By definition, $\mathcal L^{\mathrm{sat}}=\ker(\pi\circ q)$.
Since $\operatorname{Im}(\varphi)=\mathcal L$, we have
$q\circ\varphi=0$, and hence $(\pi\circ q)\circ\varphi=0$. Therefore,
$\varphi$ factors uniquely through $\ker(\pi\circ q)$, giving a
nonzero morphism $L\to\mathcal L^{\mathrm{sat}}$. This morphism defines a non-zero holomorphic section of $\mathcal{L}^{\mathrm{sat}} - L$.
\end{proof}

We use the convention that $\mathbb P_X(E)$ parametrizes
one-dimensional subspaces of the fibres of $E$.

\begin{lem}\label{lem:meromorphic_section_saturated_line}
    Let $E$ be a holomorphic vector bundle on $X$, and let
    $\sigma :  X\dashrightarrow \mathbb P_X(E)$ be a meromorphic section. Then $\sigma$ determines a unique saturated
    coherent rank-one subsheaf $\mathcal L_\sigma\subset E$. 
    
    More precisely, for a general point $x\in X$, the fibre of
    $\mathcal L_\sigma$ at $x$ is precisely the line in $E_x$
    represented by $\sigma(x) \in \mathbb P(E_x)$.
\end{lem}

\begin{proof}
    Let $\Gamma\subset\mathbb P_X(E)$ be the reduced closure of
    $\sigma(U)$, where $U\subset X$ be the maximal Zariski open subset on which $\sigma$ is holomorphic, and let
    $\nu : \widetilde X\to\Gamma$ be a resolution. Denote by
    \[
        \mu : \widetilde X\to X,
        ~~
        \widetilde\sigma : \widetilde X\to\mathbb P_X(E)
    \]
    the induced morphisms. Then
    $\pi\circ\widetilde\sigma=\mu$, where
    $\pi : \mathbb P_X(E)\to X$ is the natural projection. Pulling
    back the tautological inclusion gives
    \[
        \widetilde{\mathcal L}
        :=
        \widetilde\sigma^*
        \mathcal O_{\mathbb P_X(E)}(-1)
        \hookrightarrow\mu^*E.
    \]
    Applying $\mu_*$ and using the projection formula and
    $\mu_*\mathcal O_{\widetilde X}=\mathcal O_X$, we obtain
    \[
        \mu_*\widetilde{\mathcal L}
        \hookrightarrow
        \mu_*\mu^*E
        \simeq E.
    \]
    The sheaf $\mu_*\widetilde{\mathcal L}$ is coherent by the proper
    direct image Theorem and has rank one, since over $U$ it agrees with
    $\sigma^*\mathcal O_{\mathbb P_X(E)}(-1)$. We define
    \[
        \mathcal L_\sigma
        :=
        \bigl(\mu_*\widetilde{\mathcal L}\bigr)^{\mathrm{sat}}
        \subset E.
    \]
    This is a saturated coherent rank-one subsheaf with the required
    description over $U$.

    Finally, such a saturated extension is unique. Indeed, if
    $\mathcal L_1,\mathcal L_2\subset E$ are saturated and agree over
    $U$, then
    $(\mathcal L_1+\mathcal L_2)/\mathcal L_1$ is a torsion subsheaf of
    the torsion-free sheaf $E/\mathcal L_1$, hence vanishes. Thus
    $\mathcal L_2\subset\mathcal L_1$, and symmetry gives
    $\mathcal L_1=\mathcal L_2$.
\end{proof}



\subsubsection{Rationality of Compact Kähler Manifolds}\label{sec 2.2.2}

Let $V$ be a compact irreducible complex analytic space. A \emph{rational curve} on $V$
is the image of a nonconstant holomorphic map $\mathbb P^1\to V$. We say
that $V$ is \emph{uniruled} if a rational curve passes through a general
point of $V$. Equivalently, there exist an irreducible compact complex
space $T$ and a dominant meromorphic map
\[
    \Phi :  \mathbb P^1\times T\dashrightarrow V
\]
whose restriction to $\mathbb P^1\times\{t\}$ is nonconstant for general
$t\in T$. The space $V$ is \emph{rationally connected}, abbreviated
RC, if two general points of $V$ can be joined by a rational curve. It is
\emph{rationally chain connected} if two general points can be joined by
a connected chain of rational curves. These two notions coincide for
smooth projective varieties over $\mathbb C$; see
\cite[Chapter~IV]{Kol96}.

Let $X$ be a compact complex manifold. For a nonconstant holomorphic map $f : \mathbb P^1\to X$, write
\[
    f^*T_X\simeq
    \bigoplus_{i=1}^{\dim X}\mathcal O_{\mathbb P^1}(a_i).
\]
The curve is called \emph{free} if $a_i\geq 0$ for every $i$, and
\emph{very free} if $a_i\geq 1$ for every $i$. 

The following criterion is standard in the projective case. As we have not found a suitable reference for the Kähler case, we include a short proof.

\begin{prop}\label{prop:free_curves_kahler}
    Let $X$ be a compact Kähler manifold.
    \begin{enumerate}
        \item[1).] The manifold $X$ is uniruled if and only if it
        contains a free rational curve.
        \item[2).] The manifold $X$ is rationally connected if and
        only if it contains a very free rational curve. In this case,
        $X$ is projective.
    \end{enumerate}
\end{prop}

\begin{proof}
    Let $f : \mathbb P^1\to X$ be a rational curve. For distinct points $z_1,...,z_k\in\mathbb P^1$,
    where $k=1$ or $2$, the differential at $[f]$ of the $k$-point
    evaluation map on the local deformation space of $f$ is
    \[
        H^0(\mathbb P^1,f^*T_X)
        \to
        \bigoplus_{j=1}^k (f^*T_X)_{z_j}.
    \]
    Recall $f^*T_X\simeq\bigoplus_i\mathcal O_{\mathbb P^1}(a_i)$, this map is
    surjective if and only if $a_i\geq k-1$ for every $i$. In that
    case $H^1(\mathbb P^1,f^*T_X)=0$, so the deformation space is smooth at
    $[f]$ and the evaluation map is a submersion there; see
    \cite{Hor73}.

    For $k=1$, it follows that a free curve deforms in a dominating
    family. Conversely, applying generic smoothness to a dominating
    family of rational curves shows that the one-point evaluation map
    is a submersion at a general point of the parameter space. Hence that member is free. This
    proves~1).

    For $k=2$, a very free curve has dominant two-point evaluation, so
    $X$ is rationally connected. Conversely, a rationally connected
    compact Kähler manifold is projective by
    \cite[Corollaire, p.~212]{Cam81}, and a smooth projective rationally
    connected variety contains a very free rational curve by
    \cite[Chapter~IV]{Kol96}. This proves~(ii), including the final
    assertion.
\end{proof}

\begin{rem}
    By the proof of Proposition~\ref{prop:free_curves_kahler}, we may in fact choose the rational curves covering a uniruled (resp. rationally connected) Kähler manifold so that their general members are free (resp. very free).
\end{rem}

We next record two results concerning families with rationally connected
general fibre. The first allows us to use projective techniques for
Kähler fibrations.

\begin{prop}\label{prop:rc_fibration_projective}
(\cite[Theorem 4.1 and Corollary 4.2]{CH24})
    Let $f :  X\to Y$ be a fibration between compact Kähler
    manifolds. If a general fibre of $f$ is rationally connected, then
    $ R^if_*\mathcal O_X=0 ~~,i>0$,
    and $f$ is a projective morphism.
\end{prop}

The second is the section Theorem of Graber--Harris--Starr, which will be
our main global input over a curve.

\begin{thm}(\cite[Theorem~1.1]{GHS03})\label{thm:GHS}
    Let $C$ be a smooth projective curve and let
    $ \pi : \mathcal X\to C$
    be a surjective morphism from a smooth projective variety. If a
    general fibre of $\pi$ is rationally connected, then $\pi$ admits
    a section.
\end{thm}

We recall the rational quotient of a compact Kähler manifold. On the set of very general
points of $X$, consider the equivalence relation generated by connected
chains of rational curves. Campana's quotient Theorem gives an almost
holomorphic dominant meromorphic map with connected fibres
\[
    r_X :  X\dashrightarrow R(X)
\]
such that, for every very general point $x\in X$, the fibre of $r_X$
through $x$ is precisely its equivalence class
\cite[Theorem~2.3 and Remark~2.8]{Cam92}; see also
\cite[Theorem~3.22]{Cam04}. Here \emph{almost holomorphic} means that
there exists a Zariski open set $R^\circ\subset R(X)$ such that $r_X$
is holomorphic and proper over $R^\circ$.

By its defining property, a general fibre of $r_X$ is rationally chain
connected. It is also a smooth compact Kähler manifold, and hence is
projective by \cite[Corollaire, p.~212]{Cam81}. Since rational
connectedness and rational chain connectedness coincide for smooth
projective varieties, the general fibre of $r_X$ is rationally
connected.

The base $R(X)$ belongs to Fujiki's class $\mathcal C$ and is therefore
bimeromorphic to a compact Kähler manifold; see
\cite[Section~IV.3]{Var89}. We may thus choose a proper bimeromorphic morphism
\[
    \nu :  Y\to R(X)
\]
from a smooth compact Kähler manifold $Y$. So we may replace $R(X)$
by $Y$ and, without changing notation, assume from now on that $R(X)$
is a smooth compact Kähler manifold. Resolving the indeterminacy of
$r_X$ gives a holomorphic model
\[
    f : \widehat X\to R(X),
\]
where $\widehat X$ is a smooth compact Kähler manifold and $f$ has the
same general fibres as $r_X$. By
Proposition~\ref{prop:rc_fibration_projective}, the morphism $f$ is
projective.

We claim that $R(X)$ is not uniruled. Otherwise, let
$C\subset R(X)$ be a general rational curve and let
$\widetilde C\simeq\mathbb P^1$ be its normalization. Let $Z$ be the
irreducible component of
$\widehat X\times_{R(X)}\widetilde C$ which dominates $\widetilde C$, and take a projective resolution
$\mathcal X_C\to Z$ which is an isomorphism over a general fibre. Since
$f$ is projective, $\mathcal X_C$ is a smooth projective variety, and
the induced morphism
\[
    \pi_C : \mathcal X_C\to\widetilde C
\]
has rationally connected general fibre. As
$\widetilde C\simeq\mathbb P^1$ is rationally connected,
\cite[Corollary~1.3]{GHS03} implies that $\mathcal X_C$ is rationally
connected.

Choose distinct general points $y_1,y_2\in C$ and general points
$x_i\in r_X^{-1}(y_i) \subset X$. The preceding conclusion shows that $x_1$ and
$x_2$ are rationally chain connected in $X$. They therefore belong to
the same equivalence class, whereas $r_X(x_1)=y_1\neq y_2=r_X(x_2)$.
This contradicts the defining property that a general fibre of $r_X$ is
an entire equivalence class. Hence $R(X)$ is not uniruled.

The map $r_X$ is called the \emph{rational quotient}, or the
\emph{maximal rationally connected fibration} (MRC fibration).

The two extreme cases are worth making explicit:

   1). $X$  is not uniruled
     if and only if
    $\operatorname{id}_X :  X\to X$
     is a rational quotient;

   2). $X$  is rationally connected
   if and only if
    $ X\to \{\mathrm{pt}\}$
 is a rational quotient.

\section{Some Properties of Numerical Dimensions}

\subsection{Invariance of Numerical Dimensions}
\subsubsection{Invariance under Divisorial Zariski Decomposition}

\begin{thm}\label{thm:invariance-zariski-decomposition}
Let $X$ be a compact Kähler manifold of dimension $n$, and let
$\alpha = \{\theta\}\in H^{1,1}(X,\mathbb R)$ be a psef class with divisorial Zariski
decomposition
\[
    \alpha=Z(\alpha)+\{N(\alpha)\}.
\]
Then, for every $0\leq k\leq n$,
$  \left\langle\alpha^k\right\rangle
    =
    \left\langle Z(\alpha)^k\right\rangle$. In particular,
$ \nu(\alpha,X)=\nu\bigl(Z(\alpha),X\bigr)$ 
and $\int_X \langle\alpha^n\rangle
    =
    \int_X\langle Z(\alpha)^n\rangle$.
\end{thm}

\begin{proof}
Fix a Kähler form $\omega$ on $X$. For $t>0$, set
\[
    \alpha_t:=\alpha+t\{\omega\},~~\theta_t := \theta + t\omega \in \alpha_t,~~
    N_t:=N(\alpha_t),~~p_t:=Z(\alpha_t),
\]
and write $ N:=N(\alpha)$, $p:=Z(\alpha)$. The class $\alpha_t$ is big. Let $T_t = \theta_t + dd^c V_{\theta_t}$ be a positive current with minimal
singularities in $\alpha_t$. Then by Siu's decomposition, we have $ R_t :=T_t - [N_t] \geq 0$.
Moreover, $ R_t $ has minimal singularities in $p_t$; indeed, for any $0\leq  S_t \in p_t$, we have 
$$
0\leq S_t' := S_t + [N_t] = \theta_t + dd^c \varphi_t \in \alpha_t,
$$
hence $\varphi_t \preceq V_{\theta_t} $ because $V_{\theta_t}$ is of minimal singularity, so we obtain $S_t \preceq R_t$. Since non-pluripolar products do not charge divisors, it follows that
\[
    \left\langle\alpha_t^k\right\rangle
    = \{\langle T_t^k \rangle \}
    = \{ \langle R_t^k \rangle  \}
    = \left\langle p_t^k\right\rangle,~~ 0\leq k\leq n.
\]

Adding a Kähler class can only decrease $\nu(\cdot,D)$ along a
prime divisor $D \subset X$. Hence $N_t\leq N$ coefficientwise. Moreover, the lower
semicontinuity of minimal multiplicities gives
$\nu(\alpha_t,D)\to\nu(\alpha,D)$ for every prime divisor $D$; see
\cite[Proposition~3.5]{Bou04}. Since $N$ has finite support by
\cite[Theorem~3.12]{Bou04}, we obtain
\[
    N_t\to N,
    ~~
    p_t\to p.
\]
Furthermore, $p_t
    =p+t\{\omega\}+\{N-N_t\}$ yields $ p_t - (p + t\{\omega\})$ is psef.

 Since $p_t - (p + t\{\omega\})\to0$, we can find sequence $\{\delta_t\}_t$ such that $\delta_t\downarrow0$ and
\[
    \delta_t\{\omega\}-( p_t - p)~\text{is Kähler}.
\]
 Consequently, by the monotonicity of the movable intersection product (which agrees with the positive intersection product) \cite[Theorem 3.5]{BDPP13}, we have
\begin{equation}\label{p_t convergence ineq}
      \left\langle
            (p+t\{\omega\})^k
        \right\rangle \leq
    \left\langle p_t^k\right\rangle
    \leq
        \left\langle
            (p+\delta_t\{\omega\})^k
        \right\rangle.
\end{equation}
Here, the inequalities between $(k,k)$-classes are understood in the sense of \cite[Theorem 3.5]{BDPP13}: namely, $\gamma_1 \leq \gamma_2$ means that 
$$
\int_X\gamma_1 \cdot [u] \leq  \int_X \gamma_2 \cdot [u]
$$
for every smooth closed semipositive $(n-k,n-k)$-form $u$. Both outer terms in (\ref{p_t convergence ineq}) converge to
$\langle p^k\rangle$ as $t\downarrow0$. Therefore,
$\lim_{t\downarrow0}
    \left\langle p_t^k\right\rangle
    =
    \left\langle p^k\right\rangle$. 
    Finally, by the definition of the positive product of a psef class,
\[
\begin{split}
    \left\langle\alpha^k\right\rangle:=
    \lim_{t\downarrow0}
    \left\langle\alpha_t^k\right\rangle 
    =
    \lim_{t\downarrow0}
    \left\langle p_t^k\right\rangle=
    \left\langle p^k\right\rangle.
\end{split}
\]
This proves the asserted equality. 
The invariance of the numerical dimension and the equality of volumes follow immediately.
\end{proof}

\begin{cor}\label{cor:numerical-dimension-zero}
For every psef class $\alpha$,
$\nu(\alpha,X)=0$ if and only if
    $Z(\alpha)=0$.
In this case, $\alpha=\{N(\alpha)\}$ is the class of an exceptional effective
$\mathbb R$-divisor.
\end{cor}

\begin{proof}
Since $Z(\alpha)$ is modified nef, its degree-one positive product is
$Z(\alpha)$. Therefore,
Theorem~\ref{thm:invariance-zariski-decomposition} gives
\[
    \left\{\left\langle\alpha\right\rangle\right\}
    =Z(\alpha).
\]
Hence $\nu(\alpha,X)=0$ if and only if $Z(\alpha)=0$. The final assertion
follows from the divisorial Zariski decomposition and
\cite[Theorem~3.12]{Bou04}.
\end{proof}

\subsubsection{Invariance under Proper Bimeromorphic Morphisms}

In this section, we prove that the analytic numerical dimension is invariant under pullback by a proper bimeromorphic morphism.

\begin{thm}\label{nd_bi_invari}
  Let $\pi: \widetilde{X} \to X$ be a proper bimeromorphic morphism between compact Kähler manifolds, and let $\alpha \in H^{1,1}(X,\mathbb R)$ be psef. Then
  $ \nu(\alpha,X) = \nu(\pi^*\alpha,\widetilde{X})$.
\end{thm}

\begin{proof}
    Fix Kähler forms $\omega$ on $X$ and $\widetilde \omega$ on $\widetilde X$. Choose $T_\varepsilon \in \alpha[-\varepsilon\omega]_{\rm an}$ such that, for all $q \leq \nu(\alpha,X)$,
    $$
    A_\varepsilon := T_\varepsilon + \varepsilon\omega,
    ~~
    \int_X \langle A_\varepsilon^q \rangle \wedge \omega^{n-q} \geq c >0.
    $$
    Choose $C>0$ such that $\pi^*\omega < C\widetilde \omega$. Then
    $$
    \pi^* T_\varepsilon + \varepsilon C\widetilde \omega
    =
    \pi^*A_\varepsilon
    +
    \varepsilon(C\widetilde \omega - \pi^*\omega)
    >
    \pi^* A_\varepsilon.
    $$
    For each $\varepsilon>0$, let
    $
    0\leq T_{{\rm min},\varepsilon}
    \in
    \pi^*\alpha + \varepsilon\{C\widetilde \omega\}
    $ 
    be a closed positive current with minimal singularities; in particular,
    $
    T_{{\rm min},\varepsilon}
    \succeq
    \pi^* T_\varepsilon + \varepsilon C\widetilde \omega.
    $ 
    By Theorem \ref{DDL25 thm 33} and the multinomial expansion of the non-pluripolar product
    $$
    \left\langle
    \bigl(\pi^* T_\varepsilon+\varepsilon C\widetilde \omega\bigr)^q
    \right\rangle
    =
    \left\langle
    \bigl(\pi^* A_\varepsilon
    +\varepsilon(C\widetilde \omega-\pi^*\omega)\bigr)^q
    \right\rangle,
    $$
    we obtain
    $$
    \int_{\widetilde X}
    \langle T_{{\rm min},\varepsilon}^q\rangle
    \wedge\widetilde \omega^{n-q}
    \geq
    \int_{\widetilde X}
    \left\langle
    \bigl(\pi^* T_\varepsilon+\varepsilon C\widetilde \omega\bigr)^q
    \right\rangle
    \wedge\widetilde \omega^{n-q}
    \geq
    \int_{\widetilde X}
    \langle(\pi^* A_\varepsilon)^q\rangle
    \wedge\widetilde \omega^{n-q}.
    $$
    It follows from Lemma \ref{prop:fixed-current-pullback} that
    $
    \langle(\pi^* A_\varepsilon)^q\rangle
    =
    \pi^\star\langle A_\varepsilon^q\rangle.
    $ 
    By Lemma \ref{lem:fibre-integration}, with $c_\pi=1$, we thus obtain
    $$
    \begin{aligned}
    \int_{\widetilde X}
    \langle(\pi^* A_\varepsilon)^q\rangle
    \wedge\widetilde \omega^{n-q}
    &\geq
    C^{q-n}
    \int_{\widetilde X}
    \pi^\star\langle A_\varepsilon^q\rangle
    \wedge(\pi^*\omega)^{n-q} \\
    &=
    C^{q-n}
    \int_X
    \langle A_\varepsilon^q\rangle
    \wedge\omega^{n-q}
    \geq c'>0.
    \end{aligned}
    $$
    Combining the two inequalities above yields
    $
    \nu(\pi^*\alpha,\widetilde X)
    \geq
    \nu(\alpha,X).
    $ 

    For the reverse inequality, choose
    $ 
    R_\varepsilon
    \in
    \pi^*\alpha[-\varepsilon\widetilde \omega]_{\rm an}
    $ 
    such that, for all
    $q\leq\nu(\pi^*\alpha,\widetilde X)$,
    $$
    B_\varepsilon
    :=
    R_\varepsilon+\varepsilon\widetilde \omega,
    ~~
    \int_{\widetilde X}
    \langle B_\varepsilon^q\rangle
    \wedge\widetilde \omega^{n-q}
    \geq c>0.
    $$
    Set
    $
    G_\varepsilon'
    :=
    \pi_*B_\varepsilon\geq0$, $
    \beta:=\pi_*\{\widetilde \omega\}$.
    Choose a smooth representative $b\in\beta$, and choose $C>0$ such that
    $
    \eta:=C\omega-b>0.
    $ 
    Then
    \begin{equation}\label{G_e-G_e'}
        0\leq
        G_\varepsilon
        :=
        G_\varepsilon'+\varepsilon\eta
        \in
        \alpha+\varepsilon\{C\omega\}.
    \end{equation}
    Let $E_\pi$ denote the exceptional locus of $\pi$. Note that $G_\varepsilon'$ is smooth on
    $
    X\backslash
    \pi\bigl(E_\pi\cup{\rm Sing}(B_\varepsilon)\bigr).
    $ 
    It follows from Lemma \ref{prop:fixed-current-pullback} that
    $
    \pi^\star\langle(G_\varepsilon')^q\rangle
    =
    \langle(\pi^*G_\varepsilon')^q\rangle$. 
    On $\widetilde X\backslash E_\pi$, it is easy to see that
    $B_\varepsilon=\pi^*G_\varepsilon'$.
    Hence, by the support Theorem for currents,
    $$
    B_\varepsilon
    =
    \pi^*G_\varepsilon'
    +
    \sum_i a_{i,\varepsilon}[E_i],
    $$
    where the $E_i\subset E_\pi$ are prime exceptional divisors and
    $a_{i,\varepsilon}\in\mathbb R$. Since non-pluripolar products do not charge the relevant analytic sets, we obtain
    \begin{equation}\label{pi-G_e and B_e}
        \pi^\star\langle(G_\varepsilon')^q\rangle
        =
        \langle(\pi^*G_\varepsilon')^q\rangle
        =
        \langle B_\varepsilon^q\rangle.
    \end{equation}
    Applying Lemma \ref{prop:DS-pullback} together with
    \eqref{pi-G_e and B_e}, we have
    $$
    \int_{\widetilde X}
    \langle B_\varepsilon^q\rangle
    \wedge\widetilde \omega^{n-q}
    \leq
    C_q
    \int_X
    \langle(G_\varepsilon')^q\rangle
    \wedge\omega^{n-q}.
    $$
    By \eqref{G_e-G_e'}, we obtain
    $$
    \int_X
    \langle(G_\varepsilon')^q\rangle
    \wedge\omega^{n-q}
    \leq
    \int_X
    \langle G_\varepsilon^q\rangle
    \wedge\omega^{n-q}.
    $$
    For each $\varepsilon>0$, let
    $
    0\leq
    G_{{\rm min},\varepsilon}
    \in
    \alpha+\varepsilon\{C\omega\}
    $
    be a closed positive current with minimal singularities; in particular, $G_{{\rm min},\varepsilon}
    \succeq
    G_\varepsilon$.
    By Theorem \ref{DDL25 thm 33}, we get
    $$
    \int_X
    \langle G_\varepsilon^q\rangle
    \wedge\omega^{n-q}
    \leq
    \int_X
    \langle G_{{\rm min},\varepsilon}^q\rangle
    \wedge\omega^{n-q}.
    $$
    Combining the three inequalities above, we finally obtain
    $$
    0 < c'
    \leq
    C'
    \int_{\widetilde X}
    \langle B_\varepsilon^q\rangle
    \wedge\widetilde \omega^{n-q}
    \leq
    \int_X
    \langle G_{{\rm min},\varepsilon}^q\rangle
    \wedge\omega^{n-q}.
    $$
    This yields
    $
    \nu(\alpha,X)
    \geq
    \nu(\pi^*\alpha,\widetilde X),
    $
    and completes the proof.
\end{proof}

As a direct application of Theorem \ref{nd_bi_invari}, we can define the analytic numerical dimension for  singular spaces. We first recall the definitions of Kähler spaces and resolutions of singularities. Let $Z$ be a reduced irreducible complex space.

\begin{defi}
    We say that $Z$ is a Kähler space if it admits a Kähler form.

    Assume that $Z$ is compact. A resolution of singularities of $Z$ is a proper bimeromorphic morphism $\pi:X\to Z$,
    obtained as a finite composition of blow-ups along smooth centers, where $X$ is a compact complex manifold.
\end{defi}

By \cite{BM97}, every reduced irreducible compact complex space $Z$ admits a resolution of singularities $\pi:X\to Z$ of the above form. Furthermore, if $Z$ is a Kähler space, then \cite[Lemma 2.2]{CMM17} shows that $X$ is a compact Kähler manifold.

We can now define the analytic numerical dimension on a compact Kähler space.

\begin{defi}\label{nd_on_singular}
    Let $Z$ be a compact Kähler space, and let $\theta$ be a closed smooth $(1,1)$-form on $Z$ such that
    $$
    {\rm PSH}(Z,\theta)\neq\varnothing.
    $$
    Let $\pi:X\to Z$ be a resolution of singularities. We define the numerical dimension of $\{\theta\}$ on $Z$ by
    $$
    \nu(\{\theta\},Z)
    :=
    \nu(\pi^*\{\theta\},X).
    $$
\end{defi}

The independence of the choice of resolution follows from Theorem~\ref{nd_bi_invari}. Let us provide a brief proof.

Let
$
\pi_i:X_i\to Z$, $i=1,2$,
be two resolutions of singularities. They induce a bimeromorphic map
$$
f:X_1\dashrightarrow X_2
$$
such that $f$ is given by $\pi_1: X_1 \to Z$ and $\pi_2^{-1}: Z \dashrightarrow X_2$. By \cite[Theorem 2.1.24]{MM07}, we can find a compact Kähler manifold $X_3$ with proper bimeromorphic morphisms $\pi_i' :X_3 \to X_i$ such that $\pi_1 \circ \pi_1' = \pi_2 \circ \pi_2'$.

It follows from Theorem \ref{nd_bi_invari} that we have
$$
\nu(\pi_1^*\{\theta\},X_1) = \nu\left((\pi_1')^*\pi_1^*\{\theta\},X_3\right) = 
\nu\left((\pi_2')^*\pi_2^*\{\theta\},X_3\right) = 
\nu(\pi_2^*\{\theta\},X_2).
$$

\subsubsection{Păun's Equidimensional Descent Trick}

In this section, we use Păun’s technique for equidimensional morphisms to construct potential functions on the base manifold with sufficiently mild singularities. The descent construction below is a singular-potential version of his argument in \cite{Paun98}.

Let $Z$ be a reduced irreducible compact Kähler space, let $Y$ be a compact Kähler manifold, and let $p :  Z\to Y $ be a proper surjective equidimensional morphism of relative dimension $r$. Fix Kähler forms $\Omega$ on $Z$ and $\omega$ on $Y$. In this section, our main Theorem is
\begin{thm}\label{paun_descent}
    Let $\theta$ be a smooth closed real $(1,1)$-form on $Y$. Assume that, for every $t \in (0,t_0)$, there exists $u_t \in {\rm PSH}(Z,p^*\theta+t\Omega)$.
Then, for all sufficiently small $t$, there exists a qpsh function $\varphi_t$ on $Y$ with analytic singularities such that
\[
\varphi_t \in {\rm PSH}(Y,\theta + \delta_t\omega),~~
\delta_t \searrow 0~~~{\rm and}~~~
u_t \leq p^*\varphi_t + O(1).
\]
\end{thm}

Before proving this Theorem, we first need two Lemmas. In the first Lemma, we use the condition of equidimensional morphism to construct an auxiliary finite mapping locally.

\begin{lem}\label{lem:relative-finite-coordinates}
For every $z_0\in Z$, there exist an open neighbourhood $V\subset Z$ of
$z_0$, contained in a local embedding chart, a holomorphic map
$w=(w_1,\ldots,w_r) :  V\to\mathbb C^r$ with $w(z_0)=0$, and an open
neighbourhood $U\subset Y\times\mathbb C^r$ of $(p(z_0),0)$ such that
\[
    f=(p,w) :  V\to U
\]
is finite and surjective.
\end{lem}

\begin{proof}
Embed a neighbourhood of $z_0$ into $Y_0\times\mathbb C^N$, where
$Y_0\subset Y$ is a coordinate neighbourhood of $p(z_0)$, so that $p$ is
the first projection. Since the fibre germ $(p^{-1}(p(z_0)),z_0)$ is pure
$r$-dimensional, a generic linear projection
$\ell : \mathbb C^N\to\mathbb C^r$ has zero-dimensional fibre on this
germ at $z_0$. Thus, after replacing $\ell$ by
$\ell-\ell(z_0)$, the map $(p,\ell)$ is quasi-finite at $z_0$.
The local finite-mapping Theorem gives a finite representative
\[
    (p,\ell) :  V\to U.
\]
Its image is analytic and has dimension
$\dim Z=\dim Y+r=\dim U$; hence, after shrinking $U$ around
$(p(z_0),0)$, the map is surjective. Taking $w=\ell|_V$ proves the Lemma.
\end{proof}

Then, for finite mappings, we can construct local PSH functions in the base space by the following Lemma.

\begin{lem}\label{lem_max_psh_irr}
    Let $f: V \to U$ be finite and surjective, where $V$ is a reduced irreducible complex analytic space and $U$ is a complex manifold. If $h \in {\rm PSH}(V)$ is not identically $-\infty$, then
    $$
    H(y):= \max_{x \in f^{-1}(y)} h(x)
    $$
    is psh on $U$.
\end{lem}

\begin{proof}
Since $f$ is finite, it is proper and its fibres are finite; hence the above
maximum is well defined. We first observe that $H$ is usc.
Indeed, let $y_\nu\to y_0$ and choose $x_\nu\in f^{-1}(y_\nu)$ such that
$H(y_\nu)=h(x_\nu)$. We first pass to a subsequence along which
$H(y_\nu)$ tends to its original upper limit. After restricting to a
relatively compact neighbourhood of $y_0$, properness of $f$ allows us to
pass to a further subsequence such that
$x_\nu\to x_0\in f^{-1}(y_0)$. Therefore
\[
    \limsup_{\nu\to\infty}H(y_\nu)
    =\limsup_{\nu\to\infty}h(x_\nu)
    \leq h(x_0)
    \leq H(y_0).
\]
The same argument also shows that $H$ is locally bounded above: if
$W\Subset U$, then $f^{-1}(\overline{W}) \subset V$ compact, hence
\[
    \sup_W H\leq \sup_{f^{-1}(\overline W)}h<+\infty.
\]

Let $R\subset V$ be the analytic subset consisting of
$\operatorname{Sing}V$ together with the critical locus of
$f|_{V_{\mathrm{reg}}}$. By generic smoothness, $R$ is a proper analytic
subset of $V$. Since $f$ is finite, the proper mapping Theorem and the
dimension Theorem show that $A:=f(R)$ is a proper analytic subset of $U$. Moreover,
\[
    f :  f^{-1}(U\setminus A)\to U\setminus A
\]
is a finite unramified covering. Thus, locally on $U\setminus A$, the map
$f$ admits finitely many holomorphic inverse branches
$s_1,\ldots,s_d$, and hence
\[
    H=\max_{1\leq j\leq d}(h\circ s_j).
\]
It follows that $H$ is plurisubharmonic on $U\setminus A$. Since it is
locally bounded above near $A$, the removable-singularities Theorem for
plurisubharmonic functions gives a plurisubharmonic extension
\[
    \widetilde H(y):=
    \begin{cases}
        H(y),&y\in U\setminus A,\\[2mm]
        \displaystyle\limsup_{\substack{z\to y\\ z\in U\setminus A}}H(z),
        &y\in A.
    \end{cases}
\]
Since $H$ is usc on all of $U$, we have $\widetilde H\leq H$.

It remains to prove the reverse inequality on $A$. Fix $y_0\in A$. If
$H(y_0)=-\infty$, then the preceding inequality already gives
$\widetilde H(y_0)=H(y_0)$. Suppose therefore that $H(y_0)>-\infty$, and
choose $x_0\in f^{-1}(y_0)$ such that
\[
    h(x_0)=H(y_0).
\]
Set $B:=f^{-1}(A)$. Since $A\subsetneq U$ and $f$ is surjective, $B$ is a
proper analytic subset of $V$. The irreducibility of $V$ implies that
$V\setminus B$ is dense at $x_0$. There is a nonconstant morphism
\[
    \gamma : (\Delta,0)\to(V,x_0)
\]
such that $\gamma(\Delta^*)\subset V\backslash B$.
The function $h\circ\gamma$ is subharmonic on $\Delta$ by Theorem \ref{FN80 thm 531}, and is finite at
$0$. Then we can find a sequence 
$\tau_\nu\to0$, with $\tau_\nu\neq0$, such that
\[
    h\bigl(\gamma(\tau_\nu)\bigr)\to h(x_0).
\]
Indeed, we set $S_\nu = \{z \in \Delta:|z| = (\nu+1)^{-1}\}$ and $\tau_\nu$ satisfying   $\gamma^*h(\tau_\nu) = \sup_{ S_\nu} \gamma^*h$,
then using the submean inequality
on circles and the upper semi-continuity of $h$, we can see that $\tau_\nu$ is the desired sequence.

Now $f(\gamma(\tau_\nu))\in U\setminus A$ and
$f(\gamma(\tau_\nu))\to y_0$. Therefore
\[
\begin{aligned}
    \widetilde H(y_0)
    &\geq \limsup_{\nu\to\infty}
        H\bigl(f(\gamma(\tau_\nu))\bigr) \\
    &\geq \lim_{\nu\to\infty}h\bigl(\gamma(\tau_\nu)\bigr)
     =h(x_0)=H(y_0).
\end{aligned}
\]
Thus $\widetilde H=H$ on $U$. Since $\widetilde H$ is
plurisubharmonic, so is $H$.
\end{proof}

\begin{proof}[Proof of Theorem \ref{paun_descent}]
    Fix $t\in(0,t_0)$, to be sufficiently small, and set $s=\sqrt t$. The
assumption on $u_t$ is equivalent to
$  p^*\theta+dd^cu_t\geq-t\Omega$.                     
By Lemma~\ref{lem:relative-finite-coordinates}, every
point $z_j\in Z$ has a neighbourhood $U_j \subset Z$ and a coordinate neighbourhood $Y_j \subset Y$ of $p(z_j)$ on which
\[
    f_j=(p,w_j) :  U_j\to W_j\subset
    Y_j\times\mathbb C^r
\]
is finite and surjective; we choose coordinates such that $f_j(z_j)=0$.
Write
$$
    B_{j,a}:=\{(y_j,\xi):|y_j|^2+|\xi|^2<a\} \subset Y_j \times \mathbb C^r,
    $$
   and
   $$
    R_j : = |y_j \circ p|^2+|w_j|^2 ~~{\rm on}~U_j.
$$
Choose numbers $0<b_j<a_j$ with $\overline{B}_{j,a_j}\Subset W_j$, and let
    $D_j$ 
be the connected component of $f_j^{-1}(B_{j,a_j})$ containing $z_j$. Finiteness gives
$K_j:=\overline{D_j}\Subset U_j$ and $R_j=a_j$ on $\partial D_j$. Set
\[
    V_j:=\{z\in D_j:R_j(z)<b_j\}.
\]
We can assume that $\bigcup_jV_j = Z$. Choose smooth functions
$q_j\geq0$ on $U_j$ such that $  dd^c q_j\geq\Omega$ near $K_j$. Here $q_j$ can be obtained from the restriction of smooth strictly psh function in the local embedding. We define
\[
    \lambda_{j,t}:=tq_j+s(b_j-R_j)
\]
on $K_j$. Then $\lambda_{j,t}\geq0$ on $V_j$. On the other hand, if
$M_j:=\max_{K_j}q_j$ and $\delta_j:=a_j-b_j$, then
\[
    \lambda_{j,t}\leq tM_j-s\delta_j
    =s(sM_j-\delta_j)<0
    ~~\text{on }\partial D_j
\]
for all sufficiently small $t$, uniformly in $j$. Moreover, after
enlarging $C$ so that $dd^c|y_j| \leq C\omega$ on $V_j$ for all $j$, we get

\begin{equation}\label{long_ineq_singular}
\begin{aligned}
    p^*\theta+dd^c(u_t+\lambda_{j,t}) &= p^*\theta + dd^c u_t + tdd^cq_j - s(p^* dd^c|y_j| + dd^c|w_j|)
    \\ &\geq p^*\theta + dd^c u_t  + t\Omega - Csp^*\omega - s \sum_{\ell=1}^rdd^c|w_{j,\ell}|^2
    \\ &\geq-Cs\,p^*\omega
       -s\sum_{\ell=1}^r dd^c|w_{j,\ell}|^2~~~{\rm on}~~D_j.
\end{aligned}
\end{equation}

 Set
\[
   \Phi_{j,t}(y):=\max_{p^{-1}(y)\cap K_j}(u_t+\lambda_{j,t})~~{\rm and}~~ \Phi_t(y):=\max_j\Phi_{j,t}(y),
\]
where the maximum over the empty set is $-\infty$. 
Since the $V_j\subset K_j$ cover $Z$, the function $\Phi_t$ is defined on all of $Y$.

We claim that each $\Phi_{j,t}$ is usc. Indeed, if $y_\nu\to y$ and $z_\nu$ are maximizers, compactness of $K_j$ gives, after extraction, $z_\nu\to z\in p^{-1}(y)\cap K_j$.
Hence $\limsup_\nu\Phi_{j,t}(y_\nu)\leq(u_t+\lambda_{j,t})(z)\leq\Phi_{j,t}(y)$.
Since only finitely many $j$ occur, $\Phi_t$ is usc. Since the $V_j$ cover $Z$ and
$\lambda_{j,t}\geq0$ on $V_j$, we have $ u_t\leq p^*\Phi_t$.                               

Now we prove that $\theta+dd^c\Phi_t+Cs\omega\geq0$. Fix
$y_0\in Y$ with $\Phi_t(y_0)>-\infty$, and choose $j$ and
$z_0\in p^{-1}(y_0)\cap K_j$ realizing the defining maximum. The boundary inequality
above and the covering by the $V_j$ imply that $z_0\in D_j$. Set 
$w_0=w_j(z_0)$ and choose an irreducible component $\Sigma$ of
\[
    \{z\in D_j:w_j(z)=w_0\}
\]
through $z_0$.
After shrinking a neighbourhood $B$ of $y_0$, the map
\[
    p|_\Sigma : \Sigma\to B
\]
is finite and surjective. Indeed, it is finite because $f_j$ is finite;
the slice is defined by $r$ equations, so $\dim\Sigma\geq\dim Y$, while
finiteness gives the reverse inequality. Its image is therefore a
full-dimensional analytic subset of $Y$ and contains $B$ after shrinking.

Let $\chi$ be a local potential of $\theta+Cs\omega$ on $B$. By (\ref{long_ineq_singular}),
$u_t+\lambda_{j,t}+\chi\circ p
       +s\sum_{\ell=1}^r|w_{j,\ell}|^2$ is psh. Since $w_j$ is constant on $\Sigma$, the restriction
of $u_t+\lambda_{j,t}+\chi\circ p$ to $\Sigma$ is plurisubharmonic and is
not identically $-\infty$. Lemma~\ref{lem_max_psh_irr} therefore
shows that
\[
    H(y):=\max_{z\in p^{-1}(y)\cap\Sigma}
           \bigl(u_t(z)+\lambda_{j,t}(z)+\chi(y)\bigr)
\]
is plurisubharmonic on $B$. By the choice of $z_0$,
\[
    H\leq\Phi_t+\chi,
    ~~
    H(y_0)=(\Phi_t+\chi)(y_0).
\]
The submean inequality for $H$ therefore gives the submean inequality for
$\Phi_t+\chi$ at $y_0$ via the first inequality. At points where $\Phi_t=-\infty$ it is
automatic. Hence $\Phi_t+\chi$ is psh, proving
$ \theta+dd^c\Phi_t+Cs\omega\geq0$.

Finally, apply Theorem \ref{thm:DP04-regularization} to this current, with loss at most
$s\omega$, and choose the regularized potential $\varphi_t$ less singular
than $\Phi_t$. Then $\varphi_t$ has analytic singularities and
\[
    \theta+dd^c\varphi_t+(C+1)s\omega\geq0,~~
    \Phi_t\preceq\varphi_t.
\]
Thus the Theorem follows from $u_t \leq p^*\Phi_t$ with
 $\delta_t:=(C+1)\sqrt t\searrow0$.
\end{proof}

\subsubsection{Invariance under Fibrations}

Let $f: X \to Y$ be a holomorphic fibration between compact Kähler manifolds, and let $\alpha \in H^{1,1}(Y,\mathbb R)$ be a psef class. Fix two Kähler forms $\omega_X$ on $X$ and $\omega_Y$ on $Y$. Assume that $n = \dim X$, $m = \dim Y$, and $r = n-m$. In this section, we will show that the numerical dimension is invariant under fibrations, i.e. $\nu(f^*\alpha,X)=\nu(\alpha,Y)$.

We prove this by establishing the two inequalities separately. We begin with the easier direction.

\begin{lem}\label{easy_direction_fib_nd}
    In this setting, we have $\nu(f^*\alpha,X)\geq\nu(\alpha,Y)$.
\end{lem}

\begin{proof}
    Fix $q\leq\nu(\alpha,Y)$. Choose
    $T_\varepsilon\in\alpha[-\varepsilon\omega_Y]_{\rm an}$ on $Y$ such that
    \[
        A_\varepsilon:=T_\varepsilon+\varepsilon\omega_Y\geq0,
        ~~
        \int_Y\langle A_\varepsilon^q\rangle
        \wedge\omega_Y^{m-q}\geq c>0.
    \]
    Choose $C>0$ such that $f^*\omega_Y\leq C\omega_X$. Then
    \[
        f^*T_\varepsilon+\varepsilon C\omega_X
        =
        f^*A_\varepsilon
        +
        \varepsilon(C\omega_X-f^*\omega_Y)
        \geq f^*A_\varepsilon.
    \]
    For each $\varepsilon>0$, let
    $0\leq T_{{\rm min},\varepsilon}
        \in
        f^*\alpha+\varepsilon\{C\omega_X\}$ 
    be a closed positive current with minimal singularities. In particular,
    $ T_{{\rm min},\varepsilon}
        \succeq
        f^*T_\varepsilon+\varepsilon C\omega_X$.
    By Theorem \ref{DDL25 thm 33} and the multinomial expansion of the non-pluripolar product, we obtain
    \[
    \begin{aligned}
        \int_X
        \langle T_{{\rm min},\varepsilon}^q\rangle
        \wedge\omega_X^{n-q}
        &\geq
        \int_X
        \left\langle
        (f^*T_\varepsilon+\varepsilon C\omega_X)^q
        \right\rangle
        \wedge\omega_X^{n-q} \\
        &\geq
        \int_X
        \langle(f^*A_\varepsilon)^q\rangle
        \wedge\omega_X^{n-q}.
    \end{aligned}
    \]
    It follows from Lemma \ref{prop:fixed-current-pullback} that $  \langle(f^*A_\varepsilon)^q\rangle
        =
        f^\star\langle A_\varepsilon^q\rangle$. By Lemma \ref{lem:fibre-integration}, we thus obtain
    \[
    \begin{aligned}
        \int_X
        \langle(f^*A_\varepsilon)^q\rangle
        \wedge\omega_X^{n-q}
        &\geq
        C^{q-m}
        \int_X
        f^\star\langle A_\varepsilon^q\rangle
        \wedge(f^*\omega_Y)^{m-q}\wedge\omega_X^r \\
        &\geq
        C^{q-m}c_f
        \int_Y
        \langle A_\varepsilon^q\rangle
        \wedge\omega_Y^{m-q}
        \geq c'>0.
    \end{aligned}
    \]
    Combining the two inequalities above yields $ \nu(f^*\alpha,X)\geq q$.
    Since this holds for every $q\leq\nu(\alpha,Y)$, we conclude that $ \nu(f^*\alpha,X)\geq\nu(\alpha,Y)$.
\end{proof}

Before proceeding to prove the inequality in the other direction, we first explain the geometric construction that will be used.

Hironaka flattening \cite{Hi75} gives a proper bimeromorphic morphism
$\pi:\widetilde Y\to Y$ and the reduced main component $Z$ of
$X\times_Y\widetilde Y$ fitting into the commutative diagram
\[
\begin{tikzcd}[column sep=large, row sep=large]
\widetilde X
    \arrow[r, "\mu"]
    \arrow[rr, bend left=29, "g"]
    \arrow[dr, "\widetilde p"']
&
Z
    \arrow[r, "q"]
    \arrow[d, "p"]
&
X
    \arrow[d, "f"]
\\
&
\widetilde Y
    \arrow[r, "\pi"']
&
Y.
\end{tikzcd}
\]
Here $p$ is an equidimensional proper surjective morphism, while $\pi$ and $q$ are proper bimeromorphic morphisms. Moreover, we may assume that $\pi$ is a finite composition of blow-ups along smooth centers; see, for example, \cite[Theorem 2.1.22]{MM07}. Hence $\widetilde Y$ is also a compact Kähler manifold. Since $Z$ is a closed analytic subspace of the compact Kähler manifold $X\times\widetilde Y$, it is a compact Kähler space. We choose $\mu:\widetilde X\to Z$ to be a resolution of singularities; in particular, $\widetilde X$ is a compact Kähler manifold. Moreover,
\[
    \widetilde p:=p\circ\mu
\]
is a fibration, due to the connectivity of its general fibers. Setting $g:=q\circ\mu$, we have
\[
    f\circ g=\pi\circ\widetilde p.
\]

\begin{lem}
    For each $\varepsilon>0$, choose any
    $  R_\varepsilon
        \in
        f^*\alpha[-\varepsilon\omega_X]_{\rm an}$
    and set $B_\varepsilon:=R_\varepsilon+\varepsilon\omega_X\geq0$.
    Let $k>\nu(\alpha,Y)$. Then
    \[
        I_\varepsilon
        :=
        \int_X
        \langle B_\varepsilon^k\rangle
        \wedge\omega_X^{n-k}
        \to0.
    \]
\end{lem}

\begin{proof}
    Write $R_\varepsilon
        =
        f^*\theta+dd^c\psi_\varepsilon
        $ on $X$.
    Fix a Kähler form $\Omega$ on $Z$ such that
    $\Omega-q^*\omega_X$ is Kähler. We have
    \[
    \begin{aligned}
        p^*(\pi^*\theta)
        +
        dd^c(\psi_\varepsilon\circ q)
        &=
        q^*f^*\theta
        +
        dd^c(\psi_\varepsilon\circ q) \\
        &=
        q^*R_\varepsilon
        \geq
        -\varepsilon q^*\omega_X
        >
        -\varepsilon\Omega
        ~~\text{on }Z.
    \end{aligned}
    \]
    By Theorem \ref{paun_descent}, there exist a Kähler form
    $\widetilde{\omega_Y}$ on $\widetilde Y$, numbers
    $\delta_\varepsilon\searrow0$, and functions
    $\varphi_\varepsilon
        \in
        {\rm PSH}
        (
            \widetilde Y,
            \pi^*\theta+\delta_\varepsilon\widetilde{\omega_Y})$
    with analytic singularities such that
    $\varphi_\varepsilon\circ p+O(1)
        \geq
        \psi_\varepsilon\circ q$ \text{on } $Z$.
    Hence $\varphi_\varepsilon\circ\widetilde p
        \succeq
        \psi_\varepsilon\circ g
        ~~\text{on }\widetilde X$.

    Define the closed positive currents
    \[
    \begin{aligned}
        G_\varepsilon
        &:=
        g^*B_\varepsilon, ~~
        G_\varepsilon' :=
        G_\varepsilon
        +
        \delta_\varepsilon
        \widetilde p^*\widetilde{\omega_Y}, \\
        C_\varepsilon
        &:=
        \widetilde p^*\pi^*\theta
        +
        dd^c(\varphi_\varepsilon\circ\widetilde p)
        +
        \delta_\varepsilon
        \widetilde p^*\widetilde{\omega_Y}
        +
        \varepsilon g^*\omega_X
    \end{aligned}
    \]
    on $\widetilde X$. Here
    \[
        G_\varepsilon',C_\varepsilon
        \in
        \widetilde p^*\pi^*\alpha
        +
        \varepsilon g^*\{\omega_X\}
        +
        \delta_\varepsilon
        \widetilde p^*\{\widetilde{\omega_Y}\}.
    \]
    Since
    $ C_\varepsilon\succeq G_\varepsilon'$, 
    it follows from Theorem \ref{DDL25 thm 33} that
    \[
        \int_{\widetilde X}
        \langle G_\varepsilon'^k\rangle
        \wedge g^*\omega_X^{n-k}
        \leq
        \int_{\widetilde X}
        \langle C_\varepsilon^k\rangle
        \wedge g^*\omega_X^{n-k}.
    \]
    We also have $ \langle G_\varepsilon^k\rangle
        \leq
        \langle G_\varepsilon'^k\rangle$
    by the multilinearity of the non-pluripolar product. Furthermore, Lemma
    \ref{prop:fixed-current-pullback} gives $g^\star\langle B_\varepsilon^k\rangle
        =
        \langle G_\varepsilon^k\rangle$.
    Since $g$ is bimeromorphic morphism, Lemma \ref{lem:fibre-integration}, with
    $c_g=1$, yields
    \[
        I_\varepsilon
        =
        \int_{\widetilde X}
        \langle G_\varepsilon^k\rangle
        \wedge g^*\omega_X^{n-k}.
    \]
    Combining the inequalities above with the expression for
    $I_\varepsilon$, we obtain
    \begin{equation}\label{G_e_G_e'_C_e}
    \begin{aligned}
        I_\varepsilon
        &=
        \int_{\widetilde X}
        \langle G_\varepsilon^k\rangle
        \wedge g^*\omega_X^{n-k} \\
        &\leq
        \int_{\widetilde X}
        \langle G_\varepsilon'^k\rangle
        \wedge g^*\omega_X^{n-k} \leq
        \int_{\widetilde X}
        \langle C_\varepsilon^k\rangle
        \wedge g^*\omega_X^{n-k}.
    \end{aligned}
    \end{equation}

    Set $S_\varepsilon
        :=
        \pi^*\theta
        +
        \delta_\varepsilon\widetilde{\omega_Y}
        +
        dd^c\varphi_\varepsilon$ on $\widetilde Y$.
    Then
    $C_\varepsilon
        =
        \widetilde p^*S_\varepsilon
        +
        \varepsilon g^*\omega_X$.
        By the multilinearity of the non-pluripolar product, we obtain
    \begin{equation}\label{C_e_S_e_pull_back}
    \begin{aligned}
        \langle C_\varepsilon^k\rangle
        &=
        \sum_{j=0}^k
        \varepsilon^{k-j}
        \binom{k}{j}
        \langle(\widetilde p^*S_\varepsilon)^j\rangle
        \wedge g^*\omega_X^{k-j} \\
        &=
        \sum_{j=0}^k
        \varepsilon^{k-j}
        \binom{k}{j}
        \widetilde p^\star
        \langle S_\varepsilon^j\rangle
        \wedge g^*\omega_X^{k-j}.
    \end{aligned}
    \end{equation}

    Fix a Kähler form $\widetilde{\omega_X}$ on $\widetilde X$ and choose
    $C>0$ such that $C\widetilde{\omega_X}\geq g^*\omega_X$.
    Applying Lemma \ref{prop:DS-pullback} to $\widetilde p$, we obtain, for
    $0\leq j\leq\min\{k,m\}$,
    \[
    \begin{aligned}
        \int_{\widetilde X}
        \widetilde p^\star\langle S_\varepsilon^j\rangle
        \wedge g^*\omega_X^{n-j}
        &\leq
        C'
        \int_{\widetilde X}
        \widetilde p^\star\langle S_\varepsilon^j\rangle
        \wedge\widetilde{\omega_X}^{n-j} \\
        &\leq
        C''
        \int_{\widetilde Y}
        \langle S_\varepsilon^j\rangle
        \wedge\widetilde{\omega_Y}^{m-j}.
    \end{aligned}
    \]
    Terms with $j>m$ vanish for dimensional reasons. Combining
    (\ref{G_e_G_e'_C_e}) and (\ref{C_e_S_e_pull_back}), we obtain
    \[
    \begin{aligned}
        I_\varepsilon
        \leq
        C''
        \sum_{j=0}^{\min\{k,m\}}
        \varepsilon^{k-j}
        \binom{k}{j}
        \int_{\widetilde Y}
        \langle S_\varepsilon^j\rangle
        \wedge\widetilde{\omega_Y}^{m-j}.
    \end{aligned}
    \]

    For every $j<k$, the quantities
    $\int_{\widetilde Y}
        \langle S_\varepsilon^j\rangle
        \wedge\widetilde{\omega_Y}^{m-j}$ 
    are uniformly bounded by a cohomological constant. Since these terms are multiplied by
    $\varepsilon^{k-j}$, their contributions tend to zero.

    If $k\leq m$, then the only remaining term is the one with $j=k$. By
    Theorem \ref{nd_bi_invari}, $\nu(\pi^*\alpha,\widetilde Y)
        =
        \nu(\alpha,Y)
        <
        k$.
    Since $\delta_\varepsilon\to0$, the definition of the numerical dimension gives
    \[
        \int_{\widetilde Y}
        \langle S_\varepsilon^k\rangle
        \wedge\widetilde{\omega_Y}^{m-k}
        \to0.
    \]
    If $k>m$, every nonzero term satisfies $j\leq m<k$ and is therefore multiplied by a positive power of $\varepsilon$. Consequently, $  I_\varepsilon\to0$.
\end{proof}

\begin{thm}\label{thm:fibration-invariance}
Let $f :  X\to Y$ be a fibration between compact Kähler
manifolds, and let $\alpha\in H^{1,1}(Y,\mathbb R)$ be psef. Then $\nu(f^*\alpha,X)=\nu(\alpha,Y)$.
\end{thm}

\begin{proof}
Lemma~\ref{easy_direction_fib_nd} gives the inequality ``$\geq$''.
If $k>\nu(\alpha,Y)$, the preceding Lemma shows, uniformly over all
admissible currents in the definition of $\nu(f^*\alpha,X)$, that the
corresponding $k$-th masses tend to zero. Hence
$\nu(f^*\alpha,X)\leq\nu(\alpha,Y)$.
\end{proof}

\subsection{\texorpdfstring{$f$}{f}-degenerate divisors}\label{Sec 3.2}

Let $f: X \to Y$ be a fibration between compact Kähler manifolds.
We first recall the notion of degenerations of the divisor under fibration.

\begin{defi}\label{def:f-degenerate-divisor}
    Let $D\geq 0$ be a vertical
    $\mathbb R$-divisor on $X$. We say that $D$ is

    \noindent 1). \emph{$f$-exceptional} if $f({\rm Supp}~D) \subset Y$ has codimension at least 2.

    \noindent 2). \emph{$f$-degenerate} if
    for every prime divisor $Q\subset Y$ such that
    $Q\subset f(\operatorname{Supp}D)$, there exists a prime divisor
    $P\subset X$ satisfying 
    $$
    f(P)=Q,~~
    P\not\subset\operatorname{Supp}D.
    $$

    \noindent 3). \emph{weakly $f$-degenerate} if either $D$ is
$f$-exceptional, or there exist prime divisors $Q\subset Y$ and
$P\subset X$ such that
\[
Q\subset f(\operatorname{Supp}D),~~
f(P)=Q,~~
P\not\subset\operatorname{Supp}D.
\]
\end{defi}

Clearly, for a vertical divisor $D \geq 0$, if $D$ is $f$-exceptional, then $D$ is $f$-degenerate; if $D$ is $f$-degenerate, then $D$ is weakly $f$-degenerate.

The following proposition relates $f$-degenerate divisors to the divisorial Zariski decomposition.

\begin{prop}\label{prop:f-degenerate-divisor}
    Let $f :  X\to Y$ be a fibration between compact Kähler manifolds,
    let $\alpha\in H^{1,1}(Y,\mathbb R)$ be a psef class, and let
    $0\neq D\geq 0$ be a vertical $\mathbb R$-divisor.
    If $D$ is weakly $f$-degenerate, then at least one prime component $\Gamma$ of $D$ such that
    \[
       \Gamma \subset {\rm Supp} ~N(f^*\alpha+\{D\}) .
    \]
\end{prop}

\begin{rem}\label{rem 314}
In fact, the notion of weak $f$-degeneracy is equivalent to {\it insufficient fiber type} as defined in \cite[Chapter III.5]{Na04}.
Proposition \ref{prop:f-degenerate-divisor} extends the consequence of \cite[Corollary III.5.3]{Na04} from psef divisors on projective varieties to psef classes on compact Kähler manifolds.    

However, we must clarify: in \cite{Na04}, the author did not require the fibers of the mapping to be connected. This is incorrect for \cite[Corollary III.5.3]{Na04}, even replacing weakly $f$-degenerate with $f$-degenerate; see Example \ref{ex:disconnected-fibres-f-degenerate} below.
\end{rem}

We isolate the analytic tool used in the proof.

\begin{lem}\label{lem:localized-test-form}
    Let $M$ be a compact Kähler manifold of dimension $d$, let
    $g :  M\to Y$ be holomorphic, and suppose that
    $W:=g(M)$ is irreducible of dimension $m$. Let
    $V\subset W_{\mathrm{reg}}$ be open and let
    $C\subset g^{-1}(V)$ be an irreducible hypersurface. If, at some smooth
    point $x_0\in C$,
    \[
        \operatorname{rank}d(g|_C)_{x_0}=m,
    \]
    then there exists a smooth strongly positive $d$-closed form
     $ \Phi\in\mathcal A^{d-1,d-1}_c\bigl(g^{-1}(V)\bigr)$
    such that
    \[
        \int_C\Phi>0,
        ~~
        g^*a\wedge\Phi=0
    \]
    for every smooth $(1,1)$-form $a$ on $Y$. Its support may be confined to
    any prescribed connected component of $g^{-1}(V)$ containing $x_0$.
\end{lem}

\begin{proof}
    Let $r=d-m$. Since $C$ has dimension $d-1$ and $g|_C$ has rank $m$ at
    $x_0$, we have $r\geq 1$. Fix a Kähler form $\Omega$ on $M$, choose a
    relatively compact neighbourhood $V_0\Subset V$ of $g(x_0)$, and choose
    a volume form $\eta$ compactly supported in $V_0$.  Set
    \[
        \Phi:=\Omega^{r-1}\wedge g^*\eta,
    \]
    on $g^{-1}(V_0)$ and extend it by zero. It is not difficult to verify that $\Phi$ satisfies all the requirements.
\end{proof}

\begin{proof}[Proof of Proposition~\ref{prop:f-degenerate-divisor}]
      Write $D=\sum_{\nu=1}^s a_\nu D_\nu,~a_\nu>0$,
    and set $\beta:=f^*\alpha+\{D\}$. Suppose that no prime component of
    $D$ is the component of $N(\beta)$.

    \smallskip
    \noindent\emph{Step 1.} Set $Y_\nu:=f(D_\nu)$. These are proper irreducible analytic subsets of
    $Y$. We will choose a suitable $W := Y_\nu$. If $D$ is $f$-exceptional, we choose $W$ maximal among the $Y_\nu$. Otherwise, choose prime divisors
    $Q\subset Y$ and $R\subset X$ so that
    $$
    f(R)=Q,~~R\not\subset\operatorname{Supp}D.
    $$
    Set $W:=Q$, we also have $Q\subset Y_\nu$ for some $\nu$. Thus $W$ is again a
    maximal member of the family $\{Y_\nu\}$.

    After relabelling, write
    \[
        Y_1=\cdots=Y_{s_0}=W,
        ~~
        Y_\nu\neq W \quad (\nu>s_0).
    \]
    Maximality ensures $W\not\subset Y_\nu$ for $\nu>s_0$. Choose a point
    \[
        q\in W_{\mathrm{reg}}
        \setminus\bigcup_{\nu>s_0}(W\cap Y_\nu)
    \]
    and also in the common flat locus of the maps $D_\nu\to W$, $\nu\leq s_0$.
    After shrinking a relatively compact coordinate neighbourhood
    $U\ni q$, every local prime component of $D|_{f^{-1}(U)}$ maps onto
    $W\cap U$, and no component with image different from $W$ meets
    $f^{-1}(U)$.

    Set $c:=\operatorname{codim}_YW$ and choose coordinates
    $z$ on $U$ such that
    \[
        W\cap U=\{z_1=\cdots=z_c=0\},
        ~~
        H:=\{z_1=0\}\subset U.
    \]
    Thus $H=W\cap U$ when $c=1$, while $W\cap U\subsetneq H$ when
    $c\geq2$.

    \smallskip
    \noindent\emph{Step 2.} Set $X_U:=f^{-1}(U)$ and let $F_1,\ldots,F_N$ be the local prime
    components of
    $  f^*H=\operatorname{Div}(z_1\circ f)$
    on $X_U$. Since every component of $D|_{X_U}$ maps into $W\cap U$, we
    may write
    \[
        f^*H=\sum_{i \in I}m_iF_i,~~
        D|_{X_U}=\sum_{i\in I}a_iF_i,~~ m_i>0,\quad a_i\geq0.
    \]
    Define
    \[
        \lambda:=\max_{i\in I}\frac{a_i}{m_i}>0,
        ~~
        b_i:=\lambda m_i-a_i\geq0,
        ~~
        G:=\sum_{i \in I}b_iF_i.
    \]
    Then
    \begin{equation}\label{eq:weak-local-complementary-divisor}
        D|_{X_U}+G
        =\lambda f^*H
        =\lambda\operatorname{Div}(z_1\circ f).
    \end{equation}
    Set $I_0:=\{i \in I:b_i=0\}$, $I_+:=\{j \in I:b_j>0\}$.
    The set $I_0$ is nonempty, and every $F_i$ with $i\in I_0$ is a branch
    of $D$ mapping onto $W\cap U$.

    We will find a component $F_j,~j\in I_+$ whose image contains $W\cap U$. If
    $c=1$, the divisor $R$ fixed in Step~1 has a local branch
    $F_j\subset R\cap X_U$ mapping onto $W\cap U$. Since
    $R\not\subset\operatorname{Supp}D$, we have $a_j=0$ and hence
    $b_j=\lambda m_j>0$. If $c\geq2$, then
    \[
        H=\bigcup_{i=1}^N f(F_i).
    \]
    Indeed, every fibre over $H$ is contained in
    $\operatorname{Supp}(f^*H)=\bigcup_iF_i$. Since the sets on the right
    are closed analytic and $H$ is irreducible,
    some $F_j$ maps onto $H$. It cannot be a branch of $D$, because all
    such branches map onto $W\cap U\subsetneq H$. Thus again
    $a_j=0<b_j$.

    For $y\in W\cap U$, set
    \[
        A_y:=f^{-1}(y)\cap\bigcup_{i\in I_0}F_i,
        ~~
        B_y:=f^{-1}(y)\cap\bigcup_{j\in I_+}F_j.
    \]
    These are nonempty closed subsets of the connected fibre $f^{-1}(y)$,
    and $f^{-1}(y)=A_y\cup B_y$. Hence $A_y\cap B_y\neq\varnothing$, so
    \[
        W\cap U
        =\bigcup_{(i,j)\in I_0\times I_+}f(F_i\cap F_j).
    \]
    By irreducibility, for some $(i,j)\in I_0\times I_+$ there is an
    irreducible component
    \[
        C\subset F_i\cap F_j,
        ~~
        f(C)=W\cap U.
    \]
    Fix such a pair $(i,j)$.
    Let $\Gamma$ be the global prime component of $D$ containing $F_i$.
    Since $b_i=0$ and $b_j>0$,
    \[
        F_i\not\subset\operatorname{Supp}G,
        ~~
        C\subset F_i\cap\operatorname{Supp}G.
    \]
    The divisors $F_i$ and $F_j$ are distinct, so $C$ is a divisor on the
    local branch $F_i$. Consequently, the local defining sections of $G$
    restrict nontrivially to $F_i$ and vanish along $C$ with positive
    multiplicity.

    \smallskip
    \noindent\emph{Step 3.} Choose a resolution of singularities
    \[
        h : \widetilde\Gamma\to\Gamma\hookrightarrow X.
    \]
    Then $\widetilde\Gamma$ is a compact Kähler manifold. The local branch $F_i$
    determines a connected component $M_i\subset h^{-1}(X_U)$ with dense
    image in $F_i$. On $M_i$, the effective local pullback of $G$ contains
    a prime divisor $\widetilde C$ mapping onto $C$. Applying
    Lemma~\ref{lem:localized-test-form} to
    $g:=f\circ h$ and $\widetilde C$, we obtain a form $\Phi$ compactly
    supported in $M_i$ such that
    \begin{equation}\label{eq:weak-localized-pairing-positive}
        \langle h^*\{G\},\Phi\rangle>0,
        ~~
        h^*f^*a\wedge\Phi=0
    \end{equation}
    for every smooth $(1,1)$-form $a$ on $Y$. Since $\Phi$ is compactly supported on $M_i$, the first pairing is well-defined.

    Write $\beta=Z(\beta)+\{N(\beta)\}$, the class $Z(\beta)$ is modified nef. By \cite[Proposition 2.4]{Bou04}, its restriction $\{Z(\beta)|_\Gamma\}$ is psef; so $h^*\{Z(\beta)|_\Gamma\}$ is also psef. Since
    $\Gamma\not\subset\operatorname{Supp}N(\beta)$, the class $h^*\{N(\beta)|_\Gamma\}$ is
    represented on $\widetilde\Gamma$ by an effective $\mathbb R$-divisor
    $N_{\widetilde\Gamma}$.

    Recall that $f^*\alpha+\{D\}=Z(\beta)+\{N(\beta)\}$, we have
    $$
    h^* \{f^*(\alpha|_\Gamma)\} + h^*\{D|_{\Gamma}\}=h^*(Z(\beta)|_\Gamma+\{N(\beta)\}|_{\Gamma}).
    $$
    The base term $h^* \{f^*(\alpha|_\Gamma)\}$ pairing with $\Phi$ vanishes by
    \eqref{eq:weak-localized-pairing-positive}. Hence
    \begin{equation}\label{eq:weak-D-pairing-nonnegative}
        \langle h^*\{D|_{\Gamma}\},\Phi\rangle
        =\langle h^*(Z(\beta)|_\Gamma),\Phi\rangle
         +\int_{N_{\widetilde\Gamma}}\Phi
        \geq0.
    \end{equation}

    On the other hand, \eqref{eq:weak-local-complementary-divisor} is a
    multiple of a divisor whose Bott--Chern class is trivial. Since $\Phi$ is compactly supported in
    $h^{-1}(X_U)$, Stokes' Theorem gives
    \[
        \langle h^*\{D\},\Phi\rangle
        +\langle h^*\{G\},\Phi\rangle_{\mathrm{loc}}=0.
    \]
    By \eqref{eq:weak-localized-pairing-positive}, the first term is
    negative, contradicting \eqref{eq:weak-D-pairing-nonnegative}.
\end{proof}

Before proceeding further, we first present a key Lemma regarding the Zariski divisorial decomposition.

\begin{lem}\label{lem:remove-negative-subdivisor}
    Let $\alpha \in H^{1,1}(X,\mathbb R)$ be a psef class, and let $E \subset X$ be an effective $\mathbb R$-divisor. If $0 \leq E \leq N(\alpha)$ coefficientwise, then $\alpha - \{ E \}$ is psef and
    $$
    N(\alpha - \{E\}) = N(\alpha) - E,~~Z(\alpha - \{E\}) = Z(\alpha).
    $$
\end{lem}
\begin{proof}
    The first assertion follows from $\alpha - \{E\} = Z(\alpha) +\{N(\alpha) - E\}$ for which $N(\alpha) - E$ is effective. 

    For the second one, we only need to show that
    $$
    \nu(\alpha -\{   E\},F) = {\rm ord}_F (N(\alpha)) - {\rm ord}_F(E)
    $$
    for any prime divisor $F \subset X $. Once this equality is proved, the formula for $N$ follows from its definition; we then obtain 
    $$
    {\rm ord}_F (N(\alpha - \{E\})) = {\rm ord}_F (N(\alpha)-E),
    $$
    which yields $N(\alpha - \{E\}) = N(\alpha) - E$.
    By the definition of minimal multiplicities, it is easy to see $\nu(\{G\},F) \leq \nu([G],F) = {\rm ord}_F (G)$. By \cite[Proposition 3.7]{Bou04}, the subadditivity of minimal multiplicities, we have
    $$
    \nu(\alpha - \{E\},F) \leq \nu(Z(\alpha),F) + \nu(\{N(\alpha) - E\},F) \leq {\rm ord}_F(N(\alpha) - E).
    $$
    Conversely, by definition of $N(\alpha)$, we have
    $$
    {\rm ord}_FN(\alpha) = \nu(\alpha,F) \leq \nu(\alpha - \{E\},F) + \nu(\{E\},F) \leq \nu(\alpha - \{E\},F) + {\rm ord}_F(E).
    $$
    Thus $\nu(\alpha - \{E\},F) \geq {\rm ord}_F(N(\alpha) - E)$, proving equality for every $F$. Subtracting the two negative parts gives the last assertion of positive parts.
\end{proof}

We now state the main result of this section. It shows that the numerical dimension is invariant not only under fibrations but also under the addition of an $f$-degenerate divisor.

\begin{thm}\label{thm:f-degenerate-main}
    Let $f: X \to Y$ be a fibration between compact Kähler manifolds.
    Let $D \geq 0$ be a $f$-degenerate $\mathbb R$-divisor on $X$ and let $\alpha \in H^{1,1}(Y,\mathbb R)$ be a psef class.  Then
    \begin{align}
        N(f^*\alpha+\{D\})&\geq D~~{\rm coefficientwise},
        \label{eq:thm1-negative-part}\\
        Z(f^*\alpha+\{D\})
          &=Z(f^*\alpha),
        \label{eq:thm1-positive-part}\\
        \nu(f^*\alpha+\{D\},X)
          &=\nu(\alpha,Y).
        \label{eq:thm1-numdim}
    \end{align}
\end{thm}

\begin{proof}
    Set $\beta:=f^*\alpha+\{D\},~~
        N:=N(\beta).$ 
    Define the effective $\mathbb R$-divisor $D_0$ by
    \[
        \operatorname{mult}_F D_0
        :=\min\bigl\{\operatorname{mult}_F D,
                       \operatorname{mult}_F N\bigr\}
    \]
    for every prime component $F$ of $D$.  Thus $0\leq D_0\leq D$ and
    $D_0\leq N$.

    Suppose that $D_0\neq D$ and set the effective divisor $D_1:=D-D_0\ne0$.  Crucially,
    $D_1$ is still $f$-degenerate: if a prime divisor
    $Q\subset f(\operatorname{Supp}D_1)$ is given, then
    $Q\subset f(\operatorname{Supp}D)$, and a prime divisor missing from
    $\operatorname{Supp}D$ also misses $\operatorname{Supp}D_1$.
    Lemma~\ref{lem:remove-negative-subdivisor} gives
    \begin{equation}\label{eq:negative-part-after-removal}
        N(f^*\alpha+\{D_1\}) = N(\beta - \{D_0\})=N-D_0.
    \end{equation}
    If $F$ is a component of $D_1$, then
    $\operatorname{ord}_F (D)>\operatorname{ord}_F (N)$.  Hence
    $\operatorname{ord}_F (D_0)=\operatorname{ord}_F (N)$, and therefore
    $F \not\subset{\rm Supp}(N-D_0)$.  Thus no component of $D_1$ occurs in
    the negative part in \eqref{eq:negative-part-after-removal}, contrary
    to Proposition~\ref{prop:f-degenerate-divisor}.  Consequently
    $D_0=D$, which proves
    \[
        D\leq N(f^*\alpha+\{D\}).
    \]

    Applying Lemma~\ref{lem:remove-negative-subdivisor} once more, we obtain
    \[
        Z(f^*\alpha+\{D\})=Z(f^*\alpha).
    \]

    Finally, it follows from Theorem \ref{thm:invariance-zariski-decomposition} and Theorem~\ref{thm:fibration-invariance} that the numerical dimension depends only on the positive part of
    the Zariski divisorial decomposition, and it is invariant under
    pullback by a fibration.
    Using the equality of positive parts just proved, we obtain
    \[
        \nu(f^*\alpha+\{D\},X)
        =\nu(f^*\alpha,X)
        =\nu(\alpha,Y).
    \]
    This proves \eqref{eq:thm1-numdim} and completes the proof.
\end{proof}

The Theorem has the following two immediate applications. The first describes
the divisorial Zariski decomposition and numerical dimension of an
$f$-degenerate divisor.

\begin{cor}
    Let $f :  X\to Y$ be a fibration between compact Kähler manifolds.
    Assume that $D\geq0$ is an $f$-degenerate $\mathbb R$-divisor. Then
    $N(\{D\})=D$ and $\nu(\{D\},X)=0$.
\end{cor}
\begin{proof}
    Applying Theorem \ref{thm:f-degenerate-main} with
    $\alpha=0\in H^{1,1}(Y,\mathbb R)$ gives $N(\{D\})\geq D$. On the
    other hand, $N(\{D\})\leq D$ holds coefficientwise for every effective
    divisor $D$, proving the first equality. The second then follows from
    Corollary \ref{cor:numerical-dimension-zero}.
\end{proof}

The second application shows that an $f$-degenerate divisor cannot create
pseudo-effectivity on total space.

\begin{cor}\label{cor:alpha is psef}
    Let $f :  X\to Y$ be a fibration between compact Kähler manifolds.
    Assume that $D\geq0$ is an $f$-degenerate $\mathbb R$-divisor and that
    $\alpha\in H^{1,1}(Y,\mathbb R)$. Then $f^*\alpha+\{D\}$ is psef if
    and only if $\alpha$ is psef.
\end{cor}
\begin{proof}
    If $\alpha$ is psef, then $f^*\alpha$ is psef, and hence so is
    $f^*\alpha+\{D\}$.

    Conversely, assume that $f^*\alpha+\{D\}$ is psef. In the proof of
    inequality~\eqref{eq:thm1-negative-part} in Theorem
    \ref{thm:f-degenerate-main}, the only property used is precisely this
    pseudo-effectivity. The same argument therefore gives
    $N(f^*\alpha+\{D\})\geq D$. Consequently,
    $$
        f^*\alpha
        =Z(f^*\alpha+\{D\})+ \left(\{N(f^*\alpha+\{D\})\}-\{D\} \right),
    $$
    so $f^*\alpha$ is psef. Since pseudo-effectivity descends under
    surjective pullback by 
    $$
    f_* \left(f^*\alpha \wedge \{\omega^r\} \right)= c\alpha,~c >0,
    $$
    here $r:= \dim X - \dim Y$ and $\omega$ is a Kähler form on $X$, we obtain $\alpha$ is psef.
\end{proof}

Finally, we provide two examples. The first one is to support the argument in Remark \ref{rem 314}. The second one is to illustrate that the condition on $D$ in Theorem \ref{thm:f-degenerate-main} cannot be relaxed to weakly $f$-degenerate.

\begin{ex}\label{ex:disconnected-fibres-f-degenerate}
    Let $g :  C'\to C$ be a connected étale double cover of smooth projective curves, and set
    \[
        X:=C'\times\mathbb P^1,
        ~~
        f:=g\circ\operatorname{pr}_1 :  X\to C.
    \]
    The fibres of $f$ are disconnected. Fix $y\in C$ and write  $g^{-1}(y)=\{x_1,x_2\}$.
    Consider the prime divisor
    $   D:=\{x_1\}\times\mathbb P^1\subset X$.
     Then $D$ is a nonzero effective vertical divisor and $f(\operatorname{Supp}D)=\{y\}$.
    Moreover, the prime divisor $ P:=\{x_2\}\times\mathbb P^1$
    satisfies
    \[
        f(P)=\{y\},
        ~~
        P\not\subset\operatorname{Supp}D.
    \]
    Hence $D$ is $f$-degenerate.

    On the other hand, $D$ is nef, and therefore
    $N(\{D\})=0$.
    Taking $\alpha=0$, no prime component of $D$ is contained in
    $\operatorname{Supp}N(\{D\})$. 
    Thus Proposition~\ref{prop:f-degenerate-divisor} fails without the
    connectedness of the fibres, even when $D$ is assumed to be
    $f$-degenerate rather than merely weakly $f$-degenerate.
\end{ex}

\begin{ex}
    Let $g :  S:=\mathbb P^1\times\mathbb P^1\to\mathbb P^1 $
be the first projection. Fix a point \(p\in g^{-1}(0)\), let
\[
    \mu :  X:=\operatorname{Bl}_pS\to S,
    ~~
    f:=g\circ\mu :  X\to\mathbb P^1,
\]
and denote by \(E\) the exceptional curve and by \(C\) the strict transform
of \(g^{-1}(0)\). Thus
\[
    f^{-1}(0)=C+E.
\]
Let \(F_\infty:=f^{-1}(\infty)\) and set
\[
    D:=E+F_\infty.
\]
Then \(D\) is weakly \(f\)-degenerate: indeed,
\[
    \{0\}\subset f(\operatorname{Supp}D),
    ~~
    f(C)=\{0\},
    ~~
    C\not\subset\operatorname{Supp}D.
\]
It is not \(f\)-degenerate, since \(F_\infty\) is the unique prime divisor
mapping onto \(\{\infty\}\) and \(F_\infty\subset\operatorname{Supp}D\).

Take \(\alpha=0\), and denote by \(F\) the numerical class of a fibre of
\(f\). Since
\[
    D=F+E,~~
    F^2=F\cdot E=0,~~ E^2=-1,
\]
and \(F\) is nef. Thanks to \cite[Theorem 4.5 and 4.8]{Bou04}, the Zariski divisorial decomposition of \(\{D\}\) on surface $S$ is given by
\[
    Z(\{D\})=\{F\},
    ~~
    N(\{D\})=E.
\]
In particular, $D\not\leq N_\sigma(\{D\})$ because  \({\rm ord}_{F_\infty}N(\{D\}) = 0\). Moreover,
\[
    Z(f^*\alpha+\{D\})
    =\{F\}\neq0=Z(f^*\alpha),
\]
and
\[
    \nu(f^*\alpha+\{D\},X)
    =\nu(\{F\},X)=1
    \neq0=\nu(\alpha,\mathbb P^1).
\]
Thus none of the three conclusions of Theorem~\ref{thm:f-degenerate-main} remains true if
\(f\)-degeneracy is replaced by weak \(f\)-degeneracy.
\end{ex}

\section{Fibrations with Rationally Connected Fibers}\label{sec 4}

\subsection{Descent of Subsheaves of Cotangent Bundles}

Let $f: X\to Y$ be a fibration between compact Kähler manifolds. We assume that a general fibre of
\(f\) is rationally connected. In this subsection we work in the
nontrivial relative case
\[
    0<d:=\dim Y<\dim X,
    ~~
    r:=\dim X-\dim Y.
\]

Let us first state a Proposition showing that, a very free curve in a general fibre is free in $X$.

\begin{prop}
\label{prop:vertical_very_free_splitting}
Let $y\in Y$ be general, so that $F:=X_y$ is smooth and rationally connected. For a very free curve $h : \mathbb P^1\to F$, let $g : \mathbb P^1\to X$ denote the composite.
There are suitable integers \(a_i\geq 1\) such that
\[
    g^*T_X
    \simeq
    \mathcal O_{\mathbb P^1}^{\oplus d}
    \oplus
    \bigoplus_{i=1}^r\mathcal O_{\mathbb P^1}(a_i)
\]
and hence $g^*\Omega_X^1
    \simeq
    \mathcal O_{\mathbb P^1}^{\oplus d}
    \oplus
    \bigoplus_{i=1}^r\mathcal O_{\mathbb P^1}(-a_i)$.
Under the decomposition, the horizontal subbundle
\(g^*f^*\Omega_Y^1\) is identified with
\(\mathcal O_{\mathbb P^1}^{\oplus d}\). In particular, $g(\mathbb P^1)$ is a free rational curve in $X$.
\end{prop}

\begin{proof}
Since \(F\) is rationally connected, it indeed admits a very free rational
curve. Thus we can choose \(h: \mathbb P^1 \to F\) such that
\begin{equation}\label{split T_F}
    h^*T_F\simeq\bigoplus_{i=1}^r\mathcal O_{\mathbb P^1}(a_i),~~ a_i\geq 1.
\end{equation}

Over the submersion locus of $f$, the exact sequence of relative tangent bundle is
$$
0\to T_{X/Y}
      \to T_X
      \to f^*T_Y
      \to 0.
$$
Since $f$ is a proper submersion near \(F\) and $f\circ g:\mathbb P^1 \to Y$ is constant, pulling back the above 
sequence gives
\begin{equation}\label{T_F T_X T_Y exact sequence}
    0\to h^*T_F
      \to g^*T_X
      \to T_{Y,y}\otimes\mathcal O_{\mathbb P^1} = \mathcal O_{\mathbb P^1}^{\oplus d}
      \to 0
\end{equation}
on $\mathbb P^1$. Its extension class belongs to
\[
    \operatorname{Ext}^1
    \bigl(\mathcal O_{\mathbb P^1}^{\oplus d},h^*T_F\bigr)
    =
    H^1(\mathbb P^1,h^*T_F)^{\oplus d} = \left(\bigoplus_{i=1}^r H^1(\mathbb P^1,\mathcal{O}_{\mathbb P^1}(a_i)) \right)^{\oplus d}
    =0.
\]
The sequence (\ref{T_F T_X T_Y exact sequence}) therefore splits, combining with (\ref{split T_F}) giving the asserted decompositions of
\(g^*T_X\) and \(g^*\Omega_X^1\).
\end{proof}

Let \(m\geq 1\) and \(L\) be a psef line bundle on \(X\), i.e. $c_1(L) \in H^{1,1}(X,\mathbb R)$ is psef, and suppose that
there is an injective morphism
$\varphi :  L\to (\Omega_X^1)^{\otimes m}$. 
We set
\[
    V:=(\Omega_Y^1)^{\otimes m}.
\]

For the remainder of this subsection, we will show that $L$ descends to a rank-one subsheaf of $V$.
The argument consists of two distinct steps. We first show that the image
of \(\varphi\) is generically contained in \(f^*V\). This only requires a
covering family of vertical very free curves. We then prove that the
resulting line in \(V_y\) is constant along a general fibre \(X_y\). For
this second step, the property of rationally
connected manifolds is crucial.

\begin{lem}\label{lem:generic_horizontalisation}
The image of \(\varphi\) is generically horizontal;  more precisely,
\[
    \operatorname{Im}(\varphi)
    \subset
    f^*V
    \subset
    (\Omega_X^1)^{\otimes m}
\]
on the submersion locus  \(X^\circ \subset X\) of $f$.
\end{lem}

\begin{proof}
Let \(X^\circ\subset X\) be the submersion locus of \(f\). On \(X^\circ\)
we have the exact cotangent sequence
\[
    0\to
    H:=f^*\Omega_Y^1
    \to
    E:=\Omega_X^1
    \to
    \Omega_{X/Y}^1
    \to 0.
\]
Choose a general curve \(g: \mathbb P^1\to X\) as in
Proposition~\ref{prop:vertical_very_free_splitting}. Taking the $m$-th tensor power of the cotangent splitting gives
\[
    g^*E^{\otimes m}
    \simeq
    \mathcal O_{\mathbb P^1}^{\oplus d^m}
    \oplus
    \bigoplus_{\alpha}\mathcal O_{\mathbb P^1}(-b_\alpha),~~b_\alpha>0,
\]
where $\mathcal O_{\mathbb P^1}^{\oplus d^m}
    =
    g^*H^{\otimes m}
    =
    g^*f^*V$.

Write \(g^*L\simeq\mathcal O_{\mathbb P^1}(\lambda)\). We first claim that
\[
    \lambda=\deg g^*L = \int_{\mathbb P^1} g^* c_1(L) \geq 0.
\]
Indeed, by the proof of Proposition~\ref{prop:free_curves_kahler} for the first assertion, the class of a free rational curve
 $\{[C]\}= \{[g(\mathbb P^1)]\}$
 is movable in the sense of
\cite[Definition~1.3]{BDPP13}. Hence, since \(c_1(L)\) is
psef, we have 
$
\langle \mathbb P^1,g^*c_1(L) \rangle =\langle [C],c_1(L)\rangle
$ 
is non-negative by \cite[Proposition~1.4]{BDPP13}.

On the other hand, the restriction of \(\varphi\) is a nonzero
morphism
\begin{equation}\label{g*L and g*Em morphism}
     \mathcal O_{\mathbb P^1}(\lambda) \simeq 
     g^*L 
    \to g^*E^{\otimes m} \simeq
    \mathcal O_{\mathbb P^1}^{\oplus d^m}
    \oplus
    \bigoplus_{\alpha}\mathcal O_{\mathbb P^1}(-b_\alpha).
\end{equation}
Since every summand on the right has degree at most zero, this forces
\(\lambda\leq 0\). Hence $\lambda=0$.
Moreover, $\operatorname{Hom}
    \bigl(\mathcal O_{\mathbb P^1},\mathcal O_{\mathbb P^1}(-b_\alpha)\bigr)= 0$ yields the image
    $ {\rm Im}(g^*\varphi) \subset g^*f^*V$.

Now consider the quotient morphism
\[
    \overline{\varphi} : 
    L|_{X^\circ}
    \to
    E^{\otimes m}/H^{\otimes m}.
\]
Suppose that \(\overline{\varphi}\not\equiv0\). Since both the source
and the target are locally free on \(X^\circ\), the set
\[
    U:=
    \left\{
        x\in X^\circ
        \;\middle|\;
        \overline{\varphi}_x : 
        L_x\to
        \bigl(E^{\otimes m}/H^{\otimes m}\bigr)_x
        \text{ is nonzero}
    \right\}
\]
is a nonempty open subset of \(X^\circ\).

Since \(f|_{X^\circ}\) is a submersion, it is an open map. We may
therefore choose a general point \(y\in f(U)\). Then
\(U\cap X_y\) is a nonempty open subset of \(X_y\). We may further choose a member of very free rational curve $g : \mathbb P^1\to X_y$
 such that \(g(p)\in U\). Under the natural identification of fibres, we have
$ (g^*\overline{\varphi})_p
    =
    \overline{\varphi}_{g(p)}
    \neq0$.
In particular, \(g^*\varphi\) is nonzero. The preceding degree argument
therefore applies to this curve and gives
$ \operatorname{Im}(g^*\varphi)
    \subset
    g^*H^{\otimes m}$.
Equivalently, \(g^*\overline{\varphi}=0\), contradicting
\((g^*\overline{\varphi})_p\neq0\).

Therefore \(\overline{\varphi}=0\) on \(X^\circ\), and hence $\operatorname{Im}(\varphi)
    \subset
    H^{\otimes m}
    =
    f^*V$
on \(X^\circ\).
\end{proof}

Before turning to the main argument, we record the following technical Proposition concerning morphisms of vector bundles.

\begin{prop}\label{lem:meromorphic_factorisation}
Let \(M\) be a complex manifold, and let $ \alpha:E\to F$, $\beta:G\to F$ 
be morphisms of vector bundles on \(M\). Assume that
\(\alpha\) is generically injective.

Let \( V\subsetneq M\) be a proper analytic subset. Suppose that over
\(M\backslash V\), there exists
\[
    \gamma^\circ:G|_{M\backslash V}
    \to E|_{M\backslash V}
\]
such that  $\alpha\circ\gamma^\circ
    =
    \beta|_{M\backslash V}$. Then \(\gamma^\circ\) extends uniquely to a meromorphic morphism
\[
    \gamma:G\dashrightarrow E~~{\rm on}~~M.
\]

\end{prop}

Here the meromorphic morphism is defined by a meromorphic section of ${\rm Hom}(G,E)$.

\begin{proof}
The assertion is local on \(M\). Choose local frames and write
\(\alpha,\beta,\gamma^\circ\) as matrices $A,~B,~C^\circ$,
so that  $AC^\circ=B$
on \(M\backslash V\).

Set \(r:=\operatorname{rk}E\). Since \(\alpha\) is generically injective,
there is a set \(I\) of \(r\) row indices such that the corresponding
\(r\times r\) submatrix \(A_I\) satisfies
$\det A_I\not\equiv0$.
Let \(B_I\) be the submatrix of \(B\) obtained by taking the same rows. 
Taking the rows indexed by $I$ on both sides of $AC^\circ
=B$, we obtain $ A_I C^\circ=B_I$. Hence on \(\{\det A_I \ne 0\}\) we have
\[
    C^\circ
    =
    A_I^{-1}B_I
    =
    \frac{\operatorname{adj}(A_I)B_I}{\det A_I}.
\]
Thus the entries of \(C^\circ\) extend meromorphically. These local
extensions agree on overlaps, since they coincide with \(\gamma^\circ\)
on a dense open subset, and therefore glue to a unique meromorphic
morphism \(\gamma:G\dashrightarrow E\).
\end{proof}

We now return to the main argument and seek to descend the line bundle $L$ from $X$ to $Y$.

\begin{lem}\label{lem:constancy_general_fibre}
The horizontal factorisation of \(\varphi\) induces a meromorphic map
 $\rho :  X\dashrightarrow\mathbb P_Y(V)$.
 Moreover, for a general point \(y\in Y\), the
restriction
$\rho_y:=\rho|_{X_y} :  X_y\dashrightarrow\mathbb P(V_y)$
is constant.
\end{lem}

\begin{proof}
Consider the natural morphism
\[
    \alpha : 
    f^*V = f^*(\Omega_Y^1)^{\otimes m}
    \to
    (\Omega_X^1)^{\otimes m}.
\]
Since \(f\) is a proper submersion over \(X^\circ\), the morphism \(\alpha\)
is injective there, and hence generically injective. By
Lemma~\ref{lem:generic_horizontalisation}, the morphism \(\varphi\)
factors holomorphically through \(\alpha\) over \(X^\circ\). Applying
Proposition~\ref{lem:meromorphic_factorisation}, this factorisation
extends uniquely to a meromorphic morphism
\begin{equation}\label{L_fV_mer}
    L\dashrightarrow f^*V.
\end{equation}
It is generically nonzero because its composition with \(\alpha\) is
\(\varphi\). Therefore, at a general point \(x\in X\), its image is a
line in $(f^*V)_x=V_{f(x)}$, and hence defines
\[
    \rho(x):=
    \bigl[\operatorname{Im}(\varphi_x)\bigr]
    \in\mathbb P(V_{f(x)}).
\]
Locally, after choosing holomorphic frames of $L$ and $f^*V$, the meromorphic morphism in \eqref{L_fV_mer} is represented by a generically nonzero tuple of meromorphic functions 
$$
x \mapsto (u_1(x),...,u_N(x)).
$$
Hence its image line ${\rm Im}(\varphi_x)$ is represented by the homogeneous coordinates
$$
x \to [u_1(x),...,u_N(x)]
$$
 which define a local meromorphic map to $\mathbb P_Y(V)$. These local maps are compatible with changes of frames and therefore glue to the meromorphic map $\rho$.

We now prove its fibrewise constancy. Fix a general point \(y\in Y\)
and set $  F:=X_y$. Then \(F\subset X^\circ\) is rationally connected. By
Proposition~\ref{prop:free_curves_kahler} and its proof, \(F\) carries
a family of very free rational curves.
Let $ g:\mathbb P^1\to F$
be a general member of this family. By
Proposition~\ref{prop:vertical_very_free_splitting}, the curve \(g(\mathbb P^1)\)
is free in \(X\). The degree argument in the proof of
Lemma~\ref{lem:generic_horizontalisation} then gives
\[
    g^*L\simeq\mathcal O_{\mathbb P^1}.
\]
The horizontal factorisation is holomorphic along \(F\), and its
pullback by a general \(g\) is a nonzero morphism
\[
    \mathcal O_{\mathbb P^1}
    \to
    g^*f^*V
    =
    V_y\otimes\mathcal O_{\mathbb P^1} \simeq \mathcal{O}_{\mathbb P^1}^{\oplus d^m}.
\]
This morphism is just given by
\[
    1\longmapsto(c_1,...,c_r),
    ~~ c_i\in\mathbb C,
\]
and is therefore constant. Thus $\rho_y\circ g :  C\to\mathbb P(V_y)$ is constant.

Finally, let \(x_1,x_2\in F\) be two general points at which
\(\rho_y\) is defined. The dominance of the two-point evaluation map
allows us to choose a curve \(g(\mathbb P^1)\) as above passing through both
points. Consequently,
\[
    \rho_y(x_1)=\rho_y(x_2).
\]
Hence \(\rho_y\) is constant on a dense open subset of \(F\), and
therefore is constant as a meromorphic map.
\end{proof}

The distinction between the two Lemmas is important: a covering family
is enough to force generic horizontalisation, whereas constancy along a
fibre requires rational curves connecting two general points.

We can now prove the descent theorem.

\begin{thm}\label{thm:descent_cotangent_line}
Let $f :  X\to Y$ be a fibration between connected compact Kähler manifolds whose
general fibre is rationally connected, and assume that
\(\dim Y>0\). Let \(L\) be a pseudo-effective line bundle admitting an
injective morphism $ \varphi :  L\to(\Omega_X^1)^{\otimes m}$,
and set $V:=(\Omega_Y^1)^{\otimes m}$. 

Then there exists a unique saturated coherent rank-one subsheaf  $M\subset V$ such that, at a general point of \(X\),
\[
    \operatorname{Im}(\varphi)
    \subset
    f^*M
    \subset
    (\Omega_X^1)^{\otimes m},
\]
where the second inclusion is induced by the natural cotangent
morphism.
\end{thm}

\begin{proof}
By Lemma~\ref{lem:constancy_general_fibre}, the horizontal
factorisation of \(\varphi\) defines a meromorphic map
\[
    \rho :  X\dashrightarrow\mathbb P_Y(V)
\]
whose restriction to a general fibre of \(f\) is constant. We first
show that \(\rho\) descends to a meromorphic section of
\(\mathbb P_Y(V)\).

Let
\[
    \Gamma_\rho
    \subset
    X\times_Y\mathbb P_Y(V)
\]
be the closure of the graph of \(\rho\), and denote by
$ q : \Gamma_\rho\to\mathbb P_Y(V)$ the second projection. We know that
\[
    \Gamma_\sigma:=q(\Gamma_\rho)
    \subset\mathbb P_Y(V)
\]
is an irreducible analytic subset. Let $\pi : \mathbb P_Y(V)\to Y$
be the natural projection. Since \(\rho\) is constant along a general
fibre of \(f\), the general fibre of
$\pi|_{\Gamma_\sigma} : \Gamma_\sigma\to Y
$
consists of a single point. Thus, \(\pi|_{\Gamma_\sigma}\) is bimeromorphic, and its meromorphic inverse $Y\dashrightarrow\Gamma_\sigma$,  defines a meromorphic section
\[
    \sigma :  Y\dashrightarrow\mathbb P_Y(V)
\]
such that  $\rho=\sigma\circ f $ as meromorphic maps.

Applying Lemma~\ref{lem:meromorphic_section_saturated_line} to
\(\sigma\), we obtain a unique saturated coherent rank-one subsheaf
\[
    M:=\mathcal L_\sigma\subset V.
\]
For a general point \(y\in Y\), the fibre \(M_y\subset V_y\) is the
line represented by \(\sigma(y)\). Hence, for a general point
\(x\in X\), $ (f^*M)_x=M_{f(x)}$
is the line represented by $\sigma(f(x))=\rho(x)$. 
By the construction of \(\rho\), its image under the natural morphism
\[
    (f^*V)_x\to
    (\Omega_X^1)^{\otimes m}_x
\]
is precisely \(\operatorname{Im}(\varphi_x)\). This proves the generic
inclusion  $\operatorname{Im}(\varphi)
    \subset
    f^*M
    \subset
    (\Omega_X^1)^{\otimes m}$.

Finally, if \(M'\subset V\) is another saturated rank-one subsheaf
with the stated property, then \(M\) and \(M'\) determine the same
line in \(V_y\) for a general point \(y\in Y\). They are therefore the
saturated closures in \(V\) of the same generic line, and hence
\(M=M'\).
\end{proof}

\begin{rem}\label{non-existence}
    Suppose $Y$ is a point, so that $X$ is rationally connected. Let $L$ be a psef line bundle on $X$ and suppose that there is an injective morphism $L\to (\Omega_X^1)^{\otimes m}$.

    Choose a very free curve $g: \mathbb P^1 \to X$ so that $[g(\mathbb P^1)]$ is movable class on $X$ and
    \begin{equation}\label{target is a direct sum}
    g^*(\Omega_X^1) \simeq \bigoplus_i \mathcal O_{\mathbb P^1}(-a_i),~~a_i \geq 1.
    \end{equation}
   Pseudo-effectivity gives \(\deg g^*L\geq0\). On the other hand, the corresponding
morphism in \eqref{g*L and g*Em morphism}, whose target is (\ref{target is a direct sum}), forces \(\deg g^*L<0\), a
contradiction.  Hence no such injection exists. This means we do not need to consider the case where the base is a point in our subsequent discussion.
\end{rem}

\subsection{Pullback of Subsheaves of Cotangent Bundles}

Throughout this subsection, let $ f :  X\to Y $ be a fibration between compact Kähler manifolds whose general fibre is
rationally connected.

\subsubsection{Multiplicity-One Component over a Base Divisor}

We first record the geometric input needed below. Its proof will be given in
the next subsubsection.

\begin{thm}[Multiplicity-one component]
    \label{thm:multiplicity_one_component} Let $f: X \to Y$ be a fibration between compact Kähler manifolds whose general fibre is ratonally connected. 
    Let $P\subset Y$ be a prime divisor and write
    \[
        f^*P=\sum_j \mu_jP_j+E,
    \]
    where the $P_j$ are precisely the prime divisors of $X$ which dominate
    $P$, and every irreducible component of $E$ maps into a subset of $Y$ of
    codimension at least two. Then $\min_j\mu_j=1$.
\end{thm}

\begin{rem}
   When both \(X\) and \(Y\) are projective, choose a sufficiently general complete-intersection curve \(C\subset Y\) which meets \(P\) transversely and  \(C \cap f(\operatorname{Supp}E) = \varnothing\); we may also assume that \(X_C:=f^{-1}(C)\) is smooth. 
   
   The general fibre of \(f_C :  X_C\to C\) is rationally connected, so Theorem~\ref{thm:GHS} provides a section \(s :  C\to X_C\). For \(p\in C\cap P\) and \(x=s(p)\), the identity \(f_C\circ s=\operatorname{id}_C\) implies that \(d(f_C)_x\) is surjective, and hence the implicit function theorem shows that \(f_C^*p\) is reduced near \(x\). Since \(C\) meets \(P\) transversely, we have \(f_C^*p=(f^*P)|_{X_C}\) near \(x\), while the choice of \(C\) ensures that \(x\) lies on some \(P_j\). It follows that \(\mu_j=1\), and hence \(\min_j\mu_j=1\). 
   
   In the compact Kähler setting, one cannot expect the existence of such curves in $Y$. We shall follow the same strategy by finding, in the Hilbert--Douady space, a suitable analogue of the curve \(C\). 
\end{rem}

\smallskip

 However, proving this Theorem requires some preparatory work. We will first give an application of this Theorem that can be connected with the $f$-degenerate divisor we defined in Section \ref{Sec 3.2}.

\begin{thm}
    \label{prop:pullback_saturation_is_degenerate}
    Let $m\geq1$ and let $\mathcal M\subset (\Omega_Y^1)^{\otimes m}$ be a saturated coherent subsheaf of rank one. Consider the morphism
    \[
        f^*\mathcal M
        \to f^*(\Omega_Y^1)^{\otimes m}
        \xrightarrow{\,(df^*)^{\otimes m}\,}
        (\Omega_X^1)^{\otimes m},
    \]
    and denote the saturation of its image in $(\Omega_X^1)^{\otimes m}$ by $(f^*\mathcal M)^{\mathrm{sat}}$.
    There is a $f$-degenerate divisor $D \geq 0$ such that $(f^*\mathcal M)^{\mathrm{sat}}
        - f^*\mathcal M = D$.
\end{thm}

\begin{proof}
    By Lemma~\ref{lem:sat_eff}, there is an effective divisor $D$ such that
    $(f^*\mathcal M)^{\mathrm{sat}}- f^*\mathcal M = D$. Set $U := X \backslash {\rm Crit}(f)$, we have $f|_U$ is a flat morphism. Since $\mathcal M$
    is saturated, it is a line subbundle of $(\Omega_Y^1)^{\otimes m}$ away
    from an analytic subset $Z\subset Y$ of codimension at least two. On
    $U \backslash f^{-1}(Z)$, the morphism
    \[
        f^*\mathcal M\to(\Omega_X^1)^{\otimes m}
    \]
    is a subbundle inclusion, and hence it is already saturated there.
    Moreover, since $f|_{U}$ is flat,
    $f^{-1}(Z)\cap U$ has codimension at least two in $U$.
    As $D$ is a divisor, it follows that
    \[
        \operatorname{Supp}D\subset X\backslash U = {\rm Crit}(f).
    \]

    We claim that every effective $\mathbb R$-divisor supported in
    ${\rm Crit}(f)$ is $f$-degenerate. Indeed, let $E\geq0$ be such a
    divisor and let $P\subset Y$ be a prime divisor. By
    Theorem~\ref{thm:multiplicity_one_component}, there is a prime divisor
    $P_0\subset X$ dominating $P$ such that $\operatorname{ord}_{P_0}(f^*P)=1$.
    
    At a general point $x\in P_0$, the morphism $f|_{P_0}$ has rank
    $\dim Y-1$. If $t$ and $s$ are local equations of $P$ and $P_0$,
    respectively, then
    \[
        f^*t=su
    \]
    for a holomorphic unit $u$. Thus $d(f^*t)_x$ is nonzero in the
    normal direction to $P_0$, and consequently $df_x$ is non-degenerate. Hence
    $P_0\not\subset {\rm Crit}(f)$, so
    $P_0\not\subset\operatorname{Supp}E$. This proves that $E$ is
    $f$-degenerate. Applying this observation to $E=D$ completes the proof.
\end{proof}

The last paragraph of the proof yields the following corollary, which reveals a relation between the critical locus of a fibration and
$f$-degenerate divisors when the general fibre is rationally connected. This
implication does not hold for arbitrary fibrations.

\begin{cor}
    \label{prop:critical_divisors_are_f_degenerate}
    Let $f :  X\to Y$ be a fibration between compact Kähler manifolds
    whose general fibre is rationally connected. Then every effective $\mathbb R$-divisor
    $E_0$ satisfying
    \[
        \operatorname{Supp}E_0\subset\operatorname{Crit}(f)
    \]
    is $f$-degenerate.
\end{cor}

\subsubsection{The Tangent-Transfer Argument}

In this subsection, we prove Theorem Theorem~\ref{thm:multiplicity_one_component}. Set
$d:=\dim Y,~n:=\dim X,~r:=n-d$.
Choose a sufficiently general point $p\in P$ such that $P$ is smooth at $p$,
\[
    p\notin f(\operatorname{Supp}E),
\]
and every fibre $(P_j)_p$ has dimension $r$. We claim that $f$ is flat near
$f^{-1}(p)$. Indeed, the pullback of a local equation of $P$ is a
non-zero-divisor on the smooth manifold $X$, so $f^{-1}(P)$ is an effective
Cartier divisor of pure codimension one. Since
$p\notin f(\operatorname{Supp}E)$, every irreducible component of
$f^{-1}(p)$ is contained in some $(P_j)_p$ and hence has dimension at most
$r$. 

On the other hand, the fibre is locally cut out in the smooth $n$-fold
$X$ by the pullbacks of $d$ local parameters on $Y$, so each of its
irreducible components has dimension at least $n-d=r$. Thus $f^{-1}(p)$ is
purely $r$-dimensional. So we can shrink a neighbourhood $U$ of $p$ so that
\[
    f^{-1}(U)\to U
\]
is flat. Moreover, thanks to Proposition \ref{prop:rc_fibration_projective}, this morphism is also projective.

Fix a relative projective embedding over $U$. Openness of smooth rational
connectedness in projective families shows that
\[
    S:=\{y\in U\mid X_y\text{ is not smooth and rationally connected}\}
\]
is a proper closed analytic subset; see
\cite[Theorem~IV.3.11]{Kol96}. If a general point of $P$ does not lie in
$S$, the Theorem follows immediately.

We may therefore choose $p$ so that $P\cap U$ is the only divisorial component of
$S$ through $p$. Take coordinates $(y_1,\ldots,y_d)$ on $U$, centred at
$p$, with $P=(y_1=0)$, and choose a general transverse disc
\[
    \iota : \Delta\hookrightarrow U,~~
    y_1\circ\iota=t,~~
    y_2\circ\iota=\cdots=y_d\circ\iota=0,
\]
such that $\Delta\cap S=\{0\}$. Set
\[
    \pi :  W:=X\times_Y\Delta\to\Delta.
\]
Then $\pi$ is flat and projective, and
     $W_t$  is smooth and rationally connected for $t \in\Delta^*$.

Since $\pi :  W\to\Delta$ is flat and projective, 
\cite[\S~4, no.~7, Satz~6]{Gra60} implies that
\[
    P_{W_t}(k):=\chi\bigl(W_t,\mathcal O_{W_t}(k)\bigr)
\]
is independent of $t\in\Delta$. We denote this common Hilbert polynomial
by $Q$.

In Douady's notation, $H(\mathbb P^N)$ parametrizes compact analytic
subspaces of $\mathbb P^N$. Let
\[
    \mathcal H:=H_Q(\mathbb P^N)
    \simeq \operatorname{Hilb}_Q(\mathbb P^N)^{\mathrm{an}}
\]
be the open-and-closed locus of $H(\mathbb P^N)$ parametrizing those with Hilbert polynomial
$Q$. Let
\[
    \mathcal U\subset\mathcal H\times\mathbb P^N,~~q:= {\rm pr}_\mathcal{H}|_{\mathcal U} : \mathcal U \to \mathcal H
\]
be the universal closed analytic subspace and the first projection. Thus, if
$h\in\mathcal H$ represents an analytic subspace
$V_h\subset\mathbb P^N$, then $\mathcal U_h:=q^{-1}(h)=V_h$.

By the universal property of the Douady space, the embedded flat family
$W\subset\Delta\times\mathbb P^N$ determines a unique holomorphic classifying map
\[
    \varphi : \Delta\to\mathcal H,
    ~~
    \varphi(t)=[W_t],
\]
for which the corresponding fibre product satisfies
\begin{equation}\label{eq:W_as_Hilbert_pullback}
    W\simeq\Delta\times_{\mathcal H}\mathcal U.
\end{equation}
See \cite[\S~9.1, Remark~2, and \S~9.7, Theorem~1]{Dou66}. Let
\[
\mathcal H^{\rm rc}
:=\{h\in\mathcal H:\mathcal U_h
\text{ is smooth and rationally connected}\}.
\]
By openness of smoothness and rational connectedness in projective
families \cite[Theorem 3.11]{Kol96}, $\mathcal H^{\rm rc}$ is Zariski open. Moreover,
$\varphi(\Delta^*)\subset\mathcal H^{\rm rc}$.

The following curve-selection Lemma converts this analytic disc into a
projective curve without changing its first-order direction.

\begin{lem}\label{lem:projective_curve_selection_prescribed_tangent}
    Let $H$ be a projective complex space, let $H^\circ\subset H$ be a
    Zariski-open subset, and let
    \[
        \alpha : (\Delta,0)\to(H,h)
    \]
    be a holomorphic germ such that
    $\alpha(\Delta^*)\subset H^\circ$. If
    $ \zeta:=d\alpha_0\left(\frac{\partial}{\partial t}\right)\neq0$, 
    then there are a smooth projective curve $C$, a point $c\in C$, a local
    coordinate $s$ at $c$, and a holomorphic map $g :  C\to H$ such that
    \[
        g(c)=h,~~
        dg_c\left(\frac{\partial}{\partial s}\right)=\zeta,
        ~~ g(C)\cap H^\circ\neq\varnothing.
    \]
    In particular, a general point of $C$ is mapped into $H^\circ$.
\end{lem}

\noindent Here and below, for a possibly singular complex space $S$ and a point
$s\in S$, we denote by
\[
    T_sS
    :=
    \operatorname{Hom}_{\mathbb C}
    \bigl(\mathfrak m_{S,s}/\mathfrak m_{S,s}^2,\mathbb C\bigr)
\]
its Zariski tangent space; see
\cite[Chapter~2.1]{Fis76}. For a holomorphic map, one similarly defines the tangent map which satisfies the chain rule; see \cite[Chapter~2.6]{Fis76}.

\begin{proof}
    Let $V$ be the irreducible Zariski
    closure of $\alpha(\Delta)$. Then $V$ is projective and
    $V\not\subset H\setminus H^\circ$. Take a  resolution $\nu : \widetilde V\to V$.
    Note that $\Delta\not\subset V_{\rm sing}$, so after shrinking $\Delta$, the punctured disc $\Delta^*$ lies in the locus over which
    $\nu$ is an isomorphism. Its lift extends across the origin to a
    holomorphic germ
    \[
        \widetilde\alpha : (\Delta,0)
        \to(\widetilde V,\widetilde h).
    \]
    Indeed, take the closure of the graph of the lift over the punctured
    disc and normalize the component dominating $\Delta$; every proper
    modification of a smooth curve is an isomorphism. Set
    $\widetilde\zeta
        :=d\widetilde\alpha_0\left(\frac{\partial}{\partial t}\right)$,
    then $d\nu_{\widetilde h}(\widetilde\zeta)=\zeta$,
    which yields $\widetilde\zeta\neq0$.

    If $\dim\widetilde V=1$, take $C=\widetilde V$. If
    $\dim\widetilde V\geq2$, apply Proposition~\ref{prop:complete_intersection_prescribed_tangent} below to obtain
    a smooth irreducible complete-intersection curve
    $C\subset\widetilde V$ through $\widetilde h$ so that
    \[
        T_{\widetilde h}C=\mathbb C\widetilde\zeta,
    \]
    and
    $C \not\subset\nu^{-1}(V\backslash H^\circ)$. Composing
    $C\hookrightarrow\widetilde V\xrightarrow{\nu}V\hookrightarrow H$ and
    rescaling a local coordinate at $\widetilde h$ gives the required map
    $g$.
\end{proof}

Let us proceed by retaining the assumptions established at the beginning of this subsection. We next transfer the tangent direction from the parameter space $\mathcal H$ to the total space $X$.

\begin{lem}
    \label{lem:tangent_transfer}
    Let $C$ be a smooth curve, let $c\in C$, and let
    $g :  C\to\mathcal H$ satisfy
    \[
        g(c)=h,
        ~~
        dg_c\left(\frac{\partial}{\partial s}\right)
        =d\varphi_0\left(\frac{\partial}{\partial t}\right)=\xi.
    \]
    If the pullback family $\mathcal X:=C\times_{\mathcal H}\mathcal U\to C$
     admits a holomorphic section, then there are $z\in W_0$ and
    $v\in T_zX$ such that
    \begin{equation}        \label{eq:nonzero_transverse_derivative}
        d(f^*y_1)_z(v)=1.
    \end{equation}
    The same conclusion holds without introducing $C$ if $\xi=0$.
\end{lem}

\begin{proof}
    Let $\sigma :  C\to\mathcal X$ be a section, and let
    $\widehat\sigma :  C\to\mathcal U$ be the induced map. Then
    $q\circ\widehat\sigma=g$. Set $z:=\sigma(c)$, viewed as a point of the
    common central fibre $W_0$. For a fibre product of complex spaces, we have the following formula for the Zariski tangent space:
    \[
        T_{(0,z)}W
        =\left\{(a,u)\in T_0\Delta\oplus T_z\mathcal U
        \;\middle|\;d\varphi_0(a)=dq_z(u)\right\}.
    \]
Indeed, the local-ring description of the fibre product gives the right-exact sequence
\[
\frac{\mathfrak m_{\mathcal H,h}}{\mathfrak m_{\mathcal H,h}^2}
\xrightarrow{\,(\varphi^*,-q^*)\,}
\frac{\mathfrak m_{\Delta,0}}{\mathfrak m_{\Delta,0}^2}
\oplus
\frac{\mathfrak m_{\mathcal U,z}}{\mathfrak m_{\mathcal U,z}^2}
\to
\frac{\mathfrak m_{W,(0,z)}}{\mathfrak m_{W,(0,z)}^2}
\to 0.
\]
Dualizing gives $T_{(0,z)}W
=
T_0\Delta\times_{T_h\mathcal H}T_z\mathcal U$, which is precisely the asserted formula.
    
    Differentiating $q\circ\widehat\sigma=g$ shows that
    \[
        \left(
        \frac{\partial}{\partial t},
        d\widehat\sigma_c\left(\frac{\partial}{\partial s}\right)
        \right)
        \in T_{(0,z)}W.
    \]
    This vector projects to $\partial/\partial t$ under $d\pi$. Using $W=X\times_Y\Delta$, let $v\in T_zX$ be
    its $X$-component. Then $df_z(v)=d\iota_0\left(\frac{\partial}{\partial t}\right)$. Since $y_1\circ\iota=t$, this gives
    \[
        d(f^*y_1)_z(v)
        =dy_{1,p}\left(d\iota_0
          \left(\frac{\partial}{\partial t}\right)\right)=1.
    \]

    If $\xi=0$, choose any $z\in W_0$. In the tangent-space description
    above, the pair $(\partial/\partial t,0)$ belongs to $T_{(0,z)}W$.
    The same argument then gives \eqref{eq:nonzero_transverse_derivative}.
\end{proof}

\begin{proof}[Proof of Theorem~\ref{thm:multiplicity_one_component}]
    We first observe that it is enough to find a point
    $z\in W_0=f^{-1}(p)$ and a tangent vector $v\in T_zX$ such that $ d(f^*y_1)_z(v)=1$.

    Indeed, $f^*P$ is locally defined near $z$ by the holomorphic
    function $f^*y_1$. We then obtain 
    $d(f^*y_1)_z\neq0$. Hence the holomorphic implicit-function Theorem
    shows that $(f^*y_1=0)$ is smooth and reduced near $z$. Its unique
    local irreducible component through $z$ therefore occurs in $f^*P$
    with multiplicity one. Moreover,
    $p\notin f(\operatorname{Supp}E)$ implies this component is the
    germ of one of the prime divisors $P_j$ dominating $P$. Consequently,
    $\mu_j=1$.

    It remains to construct such a pair $(z,v)$. Set
    \[
        \xi:=d\varphi_0\left(\frac{\partial}{\partial t}\right)
        \in T_{\varphi(0)}\mathcal H.
    \]
    
    If $\xi=0$, the zero-tangent case of
    Lemma~\ref{lem:tangent_transfer} directly gives
    $z\in W_0$ and $v\in T_zX$ satisfying
    \eqref{eq:nonzero_transverse_derivative}. This completes the proof when $\xi=0$.

    Suppose now that $\xi\neq0$. Applying  Lemma~\ref{lem:projective_curve_selection_prescribed_tangent}, we obtain a smooth projective curve $C$, a point $c\in C$, a local
    coordinate $s$ at $c$, and a holomorphic map
    $g :  C\to\mathcal H$ such that
    \[
        g(c)=\varphi(0),~~
        dg_c\left(\frac{\partial}{\partial s}\right)
        =
        d\varphi_0\left(\frac{\partial}{\partial t}\right),
    \]
    and a general point of $C$ is mapped into
    $\mathcal H^{\mathrm{rc}}$.

    Consider the pullback of the universal family
    \[
        \mathcal X:=C\times_{\mathcal H}\mathcal U\to C.
    \]
    This is a projective family whose general fibre is smooth and
    rationally connected. Take its irreducible component dominating $C$
    and a resolution which is an isomorphism over the general
    smooth locus. The general fibre of the resolved family is still
    rationally connected. Theorem~\ref{thm:GHS} therefore provides a
    section of the resolved family over $C$, whose composition with the
    resolution is a section of $\mathcal X\to C$.

    Applying Lemma~\ref{lem:tangent_transfer} to this section and to the
    equality of tangent vectors above gives
    $z\in W_0$ and $v\in T_zX$ satisfying   $d(f^*y_1)_z(v)=1$.
\end{proof}

\subsubsection{The Bertini Theorem with Prescribed Tangent Direction}

In this separate subsection, we present a Proposition  stating that, given a point $x$ and a vector $v \in T_x V$ on a projective manifold $V$, one can find a curve passing through that point such that its tangent bundle is spanned by the $ v$ at $x$.

This statement is presumably standard; as we have not found a suitable reference, we include a proof.

\begin{prop}\label{prop:complete_intersection_prescribed_tangent}
    Let $V$ be a projective manifold of dimension $N\geq2$,
    let $x\in V$, let $0\neq v\in T_xV$, and let $A$ be an ample line bundle
    on $V$. For every sufficiently large integer $k$, there are
    hypersurfaces
    \[
        H_1,\ldots,H_{N-1}\in|A^k|
    \]
    such that $C:=H_1\cap\cdots\cap H_{N-1}$ is a smooth irreducible projective curve containing $x$ and $T_xC=\mathbb Cv$. Moreover, given a proper subvariety $S\subsetneq V$, the curve may be
    chosen so that $C\not\subset S$.
\end{prop}

\begin{proof}
    Choose linearly independent covectors
    \[
        \lambda_1,\ldots,\lambda_{N-1}\in T_x^*V
        ~~\text{such that}~~
        \bigcap_{i=1}^{N-1}\ker\lambda_i=\mathbb Cv.
    \]
    Let $\mathcal I_x\subset\mathcal O_V$ be the coherent ideal sheaf of
holomorphic functions vanishing at $x$. Fix $k\gg1$ and set $L:=A^k$.
Consider the exact sequence
\[
    0\to \mathcal I_x^2\otimes L
    \to L
    \to L/\mathcal I_x^2L
    \to0.
\]
By Serre vanishing $H^1(V,\mathcal I_x^2\otimes L)=0$ for $k$ sufficiently large. Hence the restriction map
$H^0(V,L)\to
    H^0(V,L/\mathcal I_x^2L)$ 
is surjective.

Choose a local holomorphic frame $e$ of $L$ near $x$. Writing $s=f_se$
locally, we have a natural identification
\[
    H^0(V,L/\mathcal I_x^2L)
    \simeq \mathbb C\oplus T_x^*V,
    ~~
    [s]\longmapsto\bigl(f_s(x),df_s(x)\bigr).
\]
Consequently, the first-jet evaluation map
\[
    j_{x,e}^1 :  H^0(V,L)
    \to\mathbb C\oplus T_x^*V,
    ~~
    s\longmapsto\bigl(f_s(x),df_s(x)\bigr),
\]
is surjective, and $\ker j_{x,e}^1
    =
    H^0(V,\mathcal I_x^2\otimes L)$.

Increasing $k$ if necessary, the linear system
$ V_0:=H^0(V,\mathcal I_x^2\otimes L)
    \subset H^0(V,L)$
generates $L$ on $U:=V\setminus\{x\}$ and induces an immersion $U\to\mathbb P(V_0^*)$;
see \cite[\begingroup\hypersetup{hidelinks}\href{https://stacks.math.columbia.edu/tag/0FD5}{\textcolor{cyan!90!blue}{0FD5}}\endgroup]{Stacks}, note that the convention for projectivization used there differs from ours by dualization.

    Choose a local frame $e$ of $L$ near $x$ and local functions $z_i$ with
    $dz_i(x)=\lambda_i$. There is a
    section $\sigma_i\in H^0(V,L)$ whose first jet at $x$ is the class of
    $z_ie$. Set
    \[
        W_i:=\mathbb C\sigma_i+V_0\subset H^0(V,L).
    \]
    If $s=a\sigma_i+u\in W_i$ with $a\neq0$ and $u\in V_0$, then a local
    equation $f_s$ of $Z(s)$ satisfies
    $ f_s=az_i+O(|z|^2)$.
    Consequently
    \begin{equation}
        df_s(x)=a\lambda_i\neq0,~~ T_xZ(s)=\ker\lambda_i.
        \label{eq:fixed_tangent_hyperplane}
    \end{equation}
    Since $V_0\subset W_i$, the system $W_i$ generates $L$ on $U$ and its
    associated morphism $U \to \mathbb P(W_i^*)$ is an immersion there.

   We construct the hypersurfaces inductively. Set $X_0:=V$. Let
$1\leq r\leq N-1$, and suppose that sections $s_i\in W_i,~i=1,\ldots,r-1$, have been chosen. Set
\[
    H_i:=\{s_i = 0\},
    ~~X_{r-1}:=H_1\cap\cdots\cap H_{r-1},
\]
and assume inductively that $X_{r-1}$ is smooth and
\[
    T_xX_{r-1}
    =
    \bigcap_{i=1}^{r-1}\ker\lambda_i.
\]
When $r=1$, these conditions are automatically satisfied. Set $U_{r-1}:=X_{r-1}\setminus\{x\}$. On the smooth quasi-projective
variety $U_{r-1}$, the restricted linear system $W_r|_{U_{r-1}}$
generates $L|_{U_{r-1}}$ and induces an immersion. Hence Bertini's
Theorem \cite[\begingroup\hypersetup{hidelinks}\href{https://stacks.math.columbia.edu/tag/0FD6}{\textcolor{cyan!90!blue}{0FD6}}\endgroup]{Stacks} shows that a general section
$s_r\in W_r$ cuts $X_{r-1}$ smoothly away from $x$. Since
$V_0\subsetneq W_r$, we may moreover choose $s_r\notin V_0$. Write
\[
    s_r=a\sigma_r+u,
    ~~
    a\neq0,\quad u\in V_0.
\]
If $f_{s_r}$ is a local equation of $H_r$ near $x$, then
\[
    f_{s_r}\equiv az_r\pmod{\mathfrak m_x^2},
    ~~
    df_{s_r}(x)=a\lambda_r.
\]
The restriction of $\lambda_r$ to $T_xX_{r-1}$ is nonzero because
$\lambda_1,\ldots,\lambda_r$ are linearly independent. Therefore $X_r:=X_{r-1}\cap H_r$ is smooth at $x$, and hence smooth everywhere, with
\[
    T_xX_r
    =
    T_xX_{r-1}\cap\ker\lambda_r
    =
    \bigcap_{i=1}^r\ker\lambda_i.
\]
Iterating this construction for $r=1,\ldots,N-1$ gives a smooth
projective curve
\[
    C:=X_{N-1}=H_1\cap\cdots\cap H_{N-1}
\]
satisfying
\[
    T_xC
    =
    \bigcap_{i=1}^{N-1}\ker\lambda_i
    =
    \mathbb Cv.
\]
Since each $X_r$ is an effective ample divisor on $X_{r-1}$,
the connectedness Theorem for ample divisors
\cite[Corollary~III.7.9]{Har77} shows inductively that $C$ is connected;
as $C$ is smooth, it is therefore irreducible.

  For the final assertion, there is nothing to prove if $x\notin S$.
Assume that $x\in S$, and choose $q\in V\setminus S$. After increasing
$k$, impose in the above construction the additional conditions
$s_i(q)=0$. By Serre vanishing, the first-order parts of the $s_i$ at
$q$ may still be chosen arbitrarily; choose them so that
\[
    ds_1(q),\ldots,ds_{N-1}(q)
\]
are linearly independent. The corresponding linear systems contain
$H^0(V,\mathcal I_x^2\mathcal I_q\otimes L)$,
which generates $L$ and induces an immersion on
$V\setminus\{x,q\}$. Hence Bertini gives smoothness away from $x$ and
$q$, while the prescribed jets at $x$ and the independence of the
differentials at $q$ give smoothness at these two points. The resulting
curve contains $q\notin S$, and therefore $C\not\subset S$.
\end{proof}

\subsection{Proof of the main Theorem}

\begin{thm}\label{thm:divisorial_descent_cotangent_line}
    Let $f :  X\to Y$ be a fibration between connected
    compact Kähler manifolds whose general fibre is rationally connected,
    and assume that $\dim Y>0$. Let $m\geq1$, and let $L$ be a psef line
    bundle admitting an injective morphism
    $\varphi :  L\to(\Omega_X^1)^{\otimes m}$.
        
    Then there exists a unique saturated coherent rank-one subsheaf
    $M\subset(\Omega_Y^1)^{\otimes m}$ such that $M$ is a psef line bundle, and for a general point
    $x\in X$,  $\varphi(L_x)
        =
        \operatorname{Im}\left(
            (f^*M)_x
            \to
            (\Omega_X^1)_x^{\otimes m}
        \right)$. Moreover, there exist effective divisors $D,E\geq0$ on $X$, where
    $D$ is $f$-degenerate, such that
    \begin{equation}\label{E and D where D f-dege}
        L\otimes\mathcal O_X(E)
        \simeq
        f^*M\otimes\mathcal O_X(D).
    \end{equation}
\end{thm}

\begin{proof}
    By Theorem~\ref{thm:descent_cotangent_line}, there exists a unique
    saturated coherent rank-one subsheaf
    $ M\subset(\Omega_Y^1)^{\otimes m}$
    such that, at a general point of $X$,
    \[
        \operatorname{Im}(\varphi)
        \subset
        f^*M
        \subset
        (\Omega_X^1)^{\otimes m}.
    \]
    Since both $\operatorname{Im}(\varphi)$ and $f^*M$ have rank one,
    the first inclusion is an equality at a general point.

    Let
    \[
        \mathcal A:=(f^*M)^{\mathrm{sat}}
        \subset(\Omega_X^1)^{\otimes m}.
    \]
    Since $\mathcal A$ is saturated, the quotient
    $(\Omega_X^1)^{\otimes m}/\mathcal A$ is torsion-free. The composite
    \[
        L\xrightarrow{\varphi}(\Omega_X^1)^{\otimes m}
        \to
        (\Omega_X^1)^{\otimes m}/\mathcal A
    \]
    vanishes at a general point because $\varphi(L)$ and $f^*M$ have
    the same generic fibre. Its image is therefore a torsion subsheaf
    of a torsion-free sheaf, so the composite vanishes identically.
    Thus $\varphi$ factors through a nonzero morphism
    \[
        L\to\mathcal A.
    \]
    As a saturated rank-one subsheaf of a vector bundle on the smooth
    manifold $X$, the sheaf $\mathcal A$ is reflexive and hence a line
    bundle. Let $E\geq0$ be the zero divisor of the above morphism. Then
    \[
        \mathcal A\simeq L\otimes\mathcal O_X(E).
    \]

    On the other hand, Proposition
    \ref{prop:pullback_saturation_is_degenerate} gives an effective
    $f$-degenerate divisor $D\geq0$ such that
    \[
        \mathcal A\simeq f^*M\otimes\mathcal O_X(D).
    \]
    Combining these two isomorphisms proves (\ref{E and D where D f-dege}). Using Corollary \ref{cor:alpha is psef} with (\ref{E and D where D f-dege}), we can conclude that $M$ is psef.
\end{proof}

We can now prove the main Theorem of this paper.

\begin{thm}\label{Main thm 1.2 in intro}
    Let $f :  X\to Y$ be a fibration between connected
    compact Kähler manifolds whose general fibre is rationally connected,
    and assume that $K_Y$ is psef. Let $m\geq1$, and let $L$ be a psef line
    bundle admitting an injective morphism
    $ \varphi :  L\to(\Omega_X^1)^{\otimes m}$.
    Then
    $$
        \nu(L,X)\leq\nu(K_Y,Y).
    $$
\end{thm}

\begin{proof}
    If $\dim Y = 0$, then according to Remark \ref{non-existence}, we know that such an $L$ does not exist. Thus we may
    assume that $\dim Y>0$.

    Let  $M\subset(\Omega_Y^1)^{\otimes m}
    $ be the saturated rank-one subsheaf given by Theorem
    \ref{thm:divisorial_descent_cotangent_line}. Since $M$ is saturated
    in a vector bundle on the smooth manifold $Y$, it is reflexive and
    hence a line bundle. Moreover,
    $$
        \mathcal Q:=(\Omega_Y^1)^{\otimes m}\big/M
    $$
    is torsion-free. The same Theorem gives effective divisors
    $D,E\geq0$, with $D$ being $f$-degenerate, such that
    \begin{equation}\label{L_E_fM_D}
        c_1(L)+\{E\}=f^*c_1(M)+\{D\}.
    \end{equation}

    Since $K_Y$ is psef, the positivity Theorem
    \cite[Corollary~5.2]{CP26}, applied to the torsion-free quotient
    $(\Omega_Y^1)^{\otimes m}\twoheadrightarrow\mathcal Q$,
    shows that $\det\mathcal Q$ is psef. Set $d:=\dim Y$ and
    $c_m:=md^{m-1}$. Since
    $\det((\Omega_Y^1)^{\otimes m})\simeq c_mK_Y$, the exact
    sequence
    $$
        0\to M
        \to(\Omega_Y^1)^{\otimes m}
        \to\mathcal Q
        \to0
    $$
    gives
    $c_m c_1(K_Y)
        =
        c_1(M)+c_1(\det\mathcal Q)$.
    Combining this identity with \eqref{L_E_fM_D}, we obtain
    $$
        c_m f^*c_1(K_Y)+\{D\}
        =
        c_1(L)+\{E\}+f^*c_1(\det\mathcal Q).
    $$
    Since $\{E\}+f^*c_1(\det\mathcal Q)$ is psef, the monotonicity of
    numerical dimension in Proposition
    \ref{prop:basic-properties-numerical-dimension} yields
    $$
        \nu\bigl(c_m f^*c_1(K_Y)+\{D\},X\bigr)
        \geq \nu(L,X).
    $$
    Finally, Theorem \ref{thm:f-degenerate-main} and the homogeneity of
    numerical dimension give
    $$
        \nu\bigl(c_m f^*c_1(K_Y)+\{D\},X\bigr)
        =
        \nu\bigl(c_m c_1(K_Y),Y\bigr)
        =
        \nu(K_Y,Y).
    $$
    The desired inequality follows.
\end{proof}

\bigskip

\noindent 
Department of Mathematics, South Kensington Campus, Imperial College London, London SW7 2AZ, United Kingdom.

\noindent
\textit{Email address:} \href{mailto:s.liu26@imperial.ac.uk}{s.liu26@imperial.ac.uk}


\begin{thebibliography}{99}

\bibitem[BM97]{BM97}
E. Bierstone, P. D. Milman.
{\it Canonical desingularization in characteristic zero by blowing up
the maximum strata of a local invariant}.
Invent. Math. 128 (1997), 207--302.

\bibitem[BCHM10]{BCHM10}
C. Birkar, P. Cascini, C. D. Hacon, J. McKernan.
{\it Existence of minimal models for varieties of log general type}.
J. Amer. Math. Soc. 23 (2010), 405--468.

\bibitem[Bir12]{Bir12}
C. Birkar.
{\it Existence of log canonical flips and a special LMMP}.
Publ. Math. IHÉS 115 (2012), 325--368.

\bibitem[Bou02]{Bou02}
S. Boucksom.
{\it On the volume of a line bundle}.
Int. J. Math. 13 (2002), 1043--1063.

\bibitem[Bou04]{Bou04}
S. Boucksom.
{\it Divisorial Zariski decompositions on compact complex manifolds}.
Ann. Sci. Éc. Norm. Supér. (4) 37 (2004), 45--76.
\bibitem[BDPP13]{BDPP13}
S. Boucksom, J.-P. Demailly, M. Păun, T. Peternell.
{\it The pseudo-effective cone of a compact Kähler manifold and
varieties of negative Kodaira dimension}.
J. Algebraic Geom. 22 (2013), 201--248.

\bibitem[BEGZ10]{BEGZ10}
S. Boucksom, P. Eyssidieux, V. Guedj, A. Zeriahi.
{\it Monge--Ampère equations in big cohomology classes}.
Acta Math. 205 (2010), 199--262.

\bibitem[Cam81]{Cam81}
F. Campana.
{\it Coréduction algébrique d'un espace analytique faiblement
Kählérien compact}.
Invent. Math. 63 (1981), 187--223.

\bibitem[Cam92]{Cam92}
F. Campana.
{\it Connexité rationnelle des variétés de Fano}.
Ann. Sci. Éc. Norm. Supér. (4) 25 (1992), 539--545.

\bibitem[Cam04]{Cam04}
F. Campana.
{\it Orbifolds, special varieties and classification theory: an
appendix}.
Ann. Inst. Fourier (Grenoble) 54 (2004), 631--665.

\bibitem[CP26]{CP26}
J. Cao, M. Păun.
{\it Remarks on relative canonical bundles and algebraicity criteria
for foliations in Kähler context}.
arXiv:2502.02183v2 (2025).

\bibitem[CH24]{CH24}
B. Claudon, A. Höring.
{\it Projectivity criteria for Kähler morphisms}.
arXiv:2404.13927 (2024).

\bibitem[CMM17]{CMM17}
D. Coman, X. Ma, G. Marinescu.
{\it Equidistribution for sequences of line bundles on normal Kähler
spaces}.
Geom. Topol. 21 (2017), 923--962.

\bibitem[DDL25]{DDL25}
T. Darvas, E. Di Nezza, H.-C. Lu.
{\it Relative pluripotential theory on compact Kähler manifolds}.
Pure Appl. Math. Q. 21 (2025), 1037--1118.

\bibitem[Dem85]{Dem85}
J.-P. Demailly.
{\it Mesures de Monge--Ampère et caractérisation géométrique des
variétés algébriques affines}.
Mém. Soc. Math. France 19 (1985), 1--125.

\bibitem[Dem12]{Dem12}
J.-P. Demailly.
{\it Complex Analytic and Differential Geometry}.
Online book, 2012, available on the author's website.

\bibitem[DP04]{DP04}
J.-P. Demailly, M. Păun.
{\it Numerical characterization of the Kähler cone of a compact
Kähler manifold}.
Ann. of Math. (2) 159 (2004), 1247--1274.

\bibitem[DS07]{DS07}
T.-C. Dinh, N. Sibony.
{\it Pull-back of currents by holomorphic maps}.
Manuscripta Math. 123 (2007), 357--371.

\bibitem[Dou66]{Dou66}
A. Douady.
{\it Le problème des modules pour les sous-espaces analytiques
compacts d'un espace analytique donné}.
Ann. Inst. Fourier (Grenoble) 16 (1966), 1--95.

\bibitem[Fis76]{Fis76}
G. Fischer.
{\it Complex Analytic Geometry}.
Lecture Notes in Mathematics 538, Springer-Verlag,
Berlin--New York, 1976.

\bibitem[FN80]{FN80}
J.-E. Fornæss, R. Narasimhan.
{\it The Levi problem on complex spaces with singularities}.
Math. Ann. 248 (1980), 47--72.

\bibitem[GL13]{GL13}
Y. Gongyo, B. Lehmann.
{\it Reduction maps and minimal model theory}.
Compos. Math. 149 (2013), 295--308.

\bibitem[GHS03]{GHS03}
T. Graber, J. Harris, J. Starr.
{\it Families of rationally connected varieties}.
J. Amer. Math. Soc. 16 (2003), 57--67.

\bibitem[Gra60]{Gra60}
H. Grauert.
{\it Ein Theorem der analytischen Garbentheorie und die Modulräume
komplexer Strukturen}.
Publ. Math. IHÉS 5 (1960), 5--64.

\bibitem[Har77]{Har77}
R. Hartshorne.
{\it Algebraic Geometry}.
Graduate Texts in Mathematics 52, Springer-Verlag,
New York--Heidelberg, 1977.

\bibitem[Hir75]{Hi75}
H. Hironaka.
{\it Flattening theorem in complex-analytic geometry}.
Amer. J. Math. 97 (1975), 503--547.

\bibitem[Hor73]{Hor73}
E. Horikawa.
{\it On deformations of holomorphic maps. I}.
J. Math. Soc. Japan 25 (1973), 372--396.

\bibitem[Kob87]{Kob87}
S. Kobayashi.
{\it Differential Geometry of Complex Vector Bundles}.
Publications of the Mathematical Society of Japan 15,
Iwanami Shoten, Tokyo, and Princeton University Press,
Princeton, 1987.

\bibitem[Kol96]{Kol96}
J. Kollár.
{\it Rational Curves on Algebraic Varieties}.
Ergebnisse der Mathematik und ihrer Grenzgebiete (3) 32,
Springer-Verlag, Berlin, 1996.

\bibitem[Leh13]{Leh13}
B. Lehmann.
{\it Comparing numerical dimensions}.
Algebra Number Theory 7 (2013), 1065--1100.

\bibitem[MM07]{MM07}
X. Ma, G. Marinescu.
{\it Holomorphic Morse Inequalities and Bergman Kernels}.
Progress in Mathematics 254, Birkhäuser, Boston, 2007.

\bibitem[MM86]{MM86}
Y. Miyaoka, S. Mori.
{\it A numerical criterion for uniruledness}.
Ann. of Math. (2) 124 (1986), 65--69.

\bibitem[Mor82]{Mor82}
S. Mori.
{\it Threefolds whose canonical bundles are not numerically effective}.
Ann. of Math. (2) 116 (1982), 133--176.

\bibitem[Nak04]{Na04}
N. Nakayama.
{\it Zariski-decomposition and abundance}.
MSJ Memoirs 14, Mathematical Society of Japan, Tokyo, 2004.

\bibitem[Ou25]{Ou25}
W. Ou.
{\it A characterization of uniruled compact Kähler manifolds}.
arXiv:2501.18088 (2025).

\bibitem[Pău98]{Paun98}
M. Păun.
{\it Sur l'effectivité numérique des images inverses de fibrés en
droites}.
Math. Ann. 310 (1998), 411--421.

\bibitem[Siu74]{Siu74}
Y.-T. Siu.
{\it Analyticity of sets associated to Lelong numbers and the
extension of closed positive currents}.
Invent. Math. 27 (1974), 53--156.

\bibitem[Stacks]{Stacks}
The Stacks Project Authors.
{\it The Stacks Project}.
\url{https://stacks.math.columbia.edu}.

\bibitem[Var89]{Var89}
J. Varouchas.
{\it Kähler spaces and proper open morphisms}.
Math. Ann. 283 (1989), 13--52.


\end{thebibliography}
\end{document}